\documentclass[12pt]{amsart}
 \usepackage{amsfonts}
\usepackage{longtable}

\usepackage{graphicx,xcolor,marginnote}%

\usepackage{amsmath,amssymb,amsthm,latexsym,tikz,anysize,enumerate}
\usepackage{graphicx,float,wrapfig,avant,colortbl,upgreek,array,tikz}

\definecolor{thistle}{rgb}{0.95,0.9,1}
 \definecolor{darkgreen}{rgb}{0,0.5,0}

\providecommand{\U}[1]{\protect\rule{.1in}{.1in}}

\newtheorem{theorem}{Theorem}[section] 
\newtheorem{definition}[theorem]{Definition}
 
\newtheorem{proposition}[theorem]{Proposition} 
\newtheorem{lemma}[theorem]{Lemma}
\newtheorem{example}[theorem]{Example}
 
\newtheorem{corollary}[theorem]{Corollary} 
\newtheorem{remark}[theorem]{Remark}

\newtheorem{construction}[theorem]{Construction}
\newtheorem{notation}[theorem]{Notation}

\newtheorem{maintheorem}{Theorem}

\def\C{{\mathcal C}}
\def\D{{\mathcal D}}
\def\PP{{\mathcal P}}
\def\B{{\mathcal B}}
\def\PG{{\rm PG}}

\def\GammaL{\Gamma{\rm L}}
\def\PGammaL{{\rm P}\Gamma{\rm L}}
\def\PIT{\mathcal{PIT}}
\def\Q{{\mathcal Q}}
\def\SS{{\mathcal S}}
\def\ff{\overline{G}}
\def\GL{{\rm GL}}

\def\la{\langle}
\def\ra{\rangle}
\def\SL{{\rm SL}}
\def\PSL{{\rm PSL}}
\def\PGL{{\rm PGL}}
\def\PGaL{{\rm P}\Gamma {\rm L}}
\def\PSigmaL{{\rm P}\Sigma {\rm L}}
\def\PSU{{\rm PSU}}
\def\SU{{\rm SU}}
\def\PGU{{\rm PGU}}
\def\GU{{\rm GU}}
\def\GammaU{\Gamma{\rm U}}
\def\PGaU{{\rm P}\Gamma {\rm U}}
\def\Sp{{\rm Sp}}

\def\GL{{\rm GL}}

\def\Om{\Omega}
\def\Si{\Sigma}
\def\Sym{{\rm Sym}}
\def\Soc{{\rm Soc}}

\def\Aut{{\rm Aut}}
\def\AGL{{\rm AGL}}
\def\al{\alpha}

\def\la{\langle}
\def\ra{\rangle}

\begin{document}

\title[Rank 3 innately transitive groups]{Rank three innately transitive permutation groups and  related $2$-transitive groups}

\thanks{The authors thank John Bamberg and David Raithel for extensive discussions on preliminary work for the project. They also thank an anonymous referee for useful feedback. This work forms part of the ARC Discovery Grant projects  
DP130100106 and DP200100080.  The first author was supported by Marsden grant 9144-3721004 (held by An - Conder - O'Brien) during the work on this paper. The third author would like to thank the Isaac Newton Institute for Mathematical Sciences, Cambridge, for support and hospitality during the programme {`Groups, representations and applications: new perspectives'},  where work on this paper was undertaken. This work was supported by EPSRC grant no EP/R014604/1.\ \ 
}
\date{today}

\author[A. A. Baykalov, A. Devillers, C. E. Praeger]{Anton A. Baykalov, Alice Devillers, and Cheryl E. Praeger}
\address{Centre for the Mathematics of Symmetry and Computation\\
School of Physics, Mathematics and Computing\\
The University of Western Australia\\
Crawley, WA 6009, Australia} 
\email{a.a.baykalov@gmail.com, alice.devillers@uwa.edu.au, cheryl.praeger@uwa.edu.au}

 \begin{abstract}
The sets of primitive, quasiprimitive, and innately transitive permutation groups may each be regarded as the building blocks of finite transitive permutation groups, and are analogues of composition factors for abstract finite groups. 
This paper extends classifications
of finite primitive and quasiprimitive groups of rank at most $3$ to a classification for the finite innately transitive groups. The new examples comprise three infinite families and three sporadic examples. A necessary step in this classification was the determination of certain configurations in finite almost simple $2$-transitive groups called special pairs.

\bigskip
\begin{center}
    Dedicated to the memory of Jacques Tits 
\end{center}
\end{abstract}

\maketitle

\section{Introduction}\label{sect:intro} 

This paper presents a classification of a family of finite permutation groups that have important geometrical significance and application, notably to describing the most symmetrical finite incidence structures called partial linear spaces. As explained in \cite{P06}, the sets of primitive, quasiprimitive, and innately transitive permutation groups may each be regarded as the building blocks of finite transitive permutation groups, and are analogues of composition factors for abstract finite groups. Considering three different subgroup lattices between a finite transitive permutation group $G$ and a point stabiliser $G_\alpha$ leads (see \cite{P06}) to three different permutational isomorphisms of $G$ with a subgroup of an iterated wreath product of the corresponding three basic kinds of permutation groups. Although the primitive groups have traditionally been regarded as the basic building blocks, both the quasiprimitive groups and the innately transitive groups have emerged as basic groups essential for describing graph symmetry as well as symmetry in finite geometry.
 The computer system Magma~\cite{magma, MAGMA1} contains a database of primitive permutation groups on sets of all sizes less than $4096$, and a similar database over the same range has recently been constructed by Bernhardt~\cite{B22} for quasiprimitive groups which are not primitive (he calls these `quimp' groups), while a similar database for innately transitive groups is more limited (\cite[Table 10.3]{B} enumerates all innately transitive groups of degree less than $60$ that are not quasiprimitive, see also Table~\ref{t:allPIT} in Section~\ref{sect:appendix}). 
 The primitive and quasiprimitive groups have also been classified by their \emph{rank}, with classifications in \cite{CAM, DGLPP, Lie87, LS86} for ranks $2$ and $3$. The purpose of this paper is to extend these classifications to innately transitive groups of rank $3$. (The {\it rank} of a transitive permutation group $G\leq \Sym(\Omega)$ on a set $\Omega$ is the number of orbits of $G$ on $\Omega\times\Omega$, or equivalently, the number of orbits in $\Omega$ of a point stabiliser. An innately transitive group that is not quasiprmitive has rank at least $3$.) 

\subsection*{Innately transitive and quasiprimitive rank $3$ permutation groups}

A permutation group is {\it innately transitive} if it has a transitive minimal normal subgroup, and such a subgroup is called a {\it plinth}. If every minimal normal subgroup is a plinth 
then the group is \emph{quasiprimitive}. As mentioned above, the finite imprimitive quasiprimitive groups of rank $3$ are known; they were classified in \cite{DGLPP}. We shall classify the rank $3$ groups which are \emph{properly innately transitive}, that is to say, innately transitive and not quasiprimitive. 

While \emph{$2$-transitive groups}, that is to say, primitive groups of rank $2$,  arise as the most highly symmetric automorphism groups of finite geometries called linear spaces, it is the innately transitive groups of rank $3$ that have this property for the more general class of partial linear spaces.  Our classification prepares the way for a classification of these highly symmetric geometries, and we discuss this, with small illustrative examples, below. We follow this by a discussion of our approach to the classification, which involves a subsidiary classification of a special subfamily of $2$-transitive permutation groups.

\subsection*{Partial linear spaces and rank $3$ permutation groups}
A \emph{partial linear space} $\D=(\PP,\B)$  is a structure with two types of objects, usually called points (elements of $\PP$)  and lines (elements of $\B$), and an incidence relation between  $\PP$ and  $\B$. We shall require the sets $\PP$ and $\B$ to be finite, each line to be incident with a constant number $k>2$ of points, each point to be incident with a constant number $r$ of lines, and each pair of  points to be incident with at most one line - and in this case the point pair is called \emph{collinear}. If each pair of points is collinear, then the geometry is a linear space: we will assume the opposite, that is, at least one pair of points is not collinear. 
The maximum degree of transitivity that can be achieved by the automorphism group of a partial linear space on pairs of points is  transitivity on (ordered) pairs of collinear points and on pairs of non-collinear points. In other words a group achieving this degree of transitivity has a rank $3$ action on points. 

The primitive rank 3 partial linear spaces are essentially classified: of the three possible types of primitive rank $3$ actions -- almost simple type, grid type, and affine type -- 
partial linear space examples for the  first two types were classified by Devillers \cite{D05,D08}, and those of affine type were recently treated in \cite{BDFP}, giving a satisfactory classification except for a few `hopeless' cases.
The problem of classifying the imprimitive rank $3$ partial linear spaces is largely untouched. Note that, to get a connected partial linear space with an imprimitive rank 3 automorphism group, each pair of collinear points must lie in distinct blocks of imprimitivity, so the number of imprimitivity blocks must be at least the line-size $k$. In the simplest case, when $k=3$ and the group has three blocks of imprimitivity,  the partial linear space is called a transversal design with $3$ blocks, and this is the only case that is classified so far:  Devillers and Hall \cite{DH06} showed that these spaces arise from multiplication tables of elementary abelian groups. There is little hope of a complete classification  of rank $3$ partial linear spaces as imprimitive rank 3 groups are so `wild'. 

Nevertheless our classification of innately transitive rank $3$ groups will, we hope, lay the groundwork for a future classification of the innately transitive rank 3 partial linear spaces.   
There are some beautiful examples which do not arise in any existing classification, and we mention two small rank $3$  examples. The first one,  Example~\ref{ex:pls1}, admits two rank $3$ groups, one of which is quasiprimitive and the other is properly innately transitive. The second one, Example~\ref{ex:pls2}, has a properly innately transitive rank $3$ group with no rank $3$ quasiprimitive subgroup, so this example will not arise if we consider only rank $3$ quasiprimitive automorphism groups.

\begin{example}\label{ex:pls1}
{\rm
{\bf [An innately transitive and  quasiprimitive example]} The point set of this
partial linear space, which appears in Figure~\ref{tikz}, is $\mathbb{Z}_{14}$ and the lines are defined by the sets of the form
$\{x,x+1,x+4,x+6\}$, $x\in\mathbb{Z}_{14}$.  The automorphism  group of this design is $C_2\times
\PSL(3,2)$, which is innately transitive of rank $3$. Its plinth $\PSL(3,2)$ acts transitively and quasiprimitively, also of rank $3$, on the $14$ points.
}
\end{example}

\begin{figure}[h]
\begin{tikzpicture}[scale =0.7,line width=1pt]
\newcommand{\twoarcs}{
\node (v1) at (0:2cm) [ball color=darkgreen, circle, draw=black, inner sep=1pt] {};  
\node (v2) at (360/7:2cm) [ball color=darkgreen, circle, draw=black, inner sep=1pt] {};
\node (v6) at (360*5/7:2cm) [ball color=darkgreen, circle, draw=black, inner sep=1pt] {};
\node (v7) at (360*6/7:2cm) [ball color=darkgreen, circle, draw=black, inner sep=1pt] {};  

\begin{scope}[rotate=-360/7]
\node (w1) at (0:5cm) [ball color=darkgreen, circle, draw=black, inner sep=1pt] {};  
\node (w2) at (360/7:5cm) [ball color=darkgreen, circle, draw=black, inner sep=1pt] {};
\node (w4) at (360*3/7:5cm) [ball color=darkgreen, circle, draw=black, inner sep=1pt] {};
\node (w5) at (4*360/7:5cm) [ball color=darkgreen, circle, draw=black, inner sep=1pt] {};
\end{scope}

\node (c11) at (15:40mm) {};
\node (c12) at (40:40mm) {};
\node (c21) at (120:4mm) {};
\node (c22) at (210:6mm) {};
\node (c31) at (270:30mm) {};
\node (c32) at (290:40mm) {};

\node (d11) at (150:58mm) {};
\node (d12) at (120:53mm) {};
\node (d21) at (50:80mm) {};
\node (d22) at (30:70mm) {};
\node (d31) at (335:45mm) {};
\node (d32) at (320:30mm) {};

\tikzset{
    electron/.style={draw=blue}, 
    gluon/.style={draw=purple}, 
}
\draw[electron] (v1) ..  controls (c11) and (c12)..  (v2);
\draw[color=blue] (v2) ..  controls (c21) and (c22) ..  (v6);
\draw[electron] (v6) ..  controls (c31) and (c32) ..  (w1);

\draw[gluon] (w5) ..  controls (d11) and (d12)..  (w4);
\draw[color=purple] (w4) ..  controls (d21) and (d22) ..  (w2);
\draw[gluon] (w2) ..  controls (d31) and (d32) ..  (v1);
}

\foreach \x in {0,1,2,3,4,5,6}{
\begin{scope}[rotate=360*\x/7]
\twoarcs
\end{scope}
}
  {
\node (v1) at (0:2cm) [ball color=thistle!10, circle, draw=thistle, inner sep=1pt,minimum width=15pt] {1};  
\node (v2) at (360/7:2cm) [ball color=thistle!10, circle, draw=thistle, inner sep=1pt,minimum width=15pt] {0};
\node (v3) at (2*360/7:2cm) [ball color=thistle!10, circle, draw=thistle, inner sep=1pt,minimum width=15pt] {11};
\node (v4) at (3*360/7:2cm) [ball color=thistle!10, circle, draw=thistle, inner sep=1pt,minimum width=15pt] {3};
\node (v5) at (4*360/7:2cm) [ball color=thistle!10, circle, draw=thistle, inner sep=1pt,minimum width=15pt] {5};
\node (v6) at (5*360/7:2cm) [ball color=thistle!10, circle, draw=thistle, inner sep=1pt,minimum width=15pt] {6};
\node (v7) at (6*360/7:2cm) [ball color=thistle!10, circle, draw=thistle, inner sep=1pt,minimum width=15pt] {2};  

\begin{scope}[rotate=-360/7]
\node (w1) at (0:5cm) [ball color=thistle!10, circle, draw=thistle, inner sep=1pt,minimum width=15pt] {4};  
\node (w2) at (360/7:5cm) [ball color=thistle!10, circle, draw=thistle, inner sep=1pt,minimum width=15pt] {10};
\node (w3) at (2*360/7:5cm) [ball color=thistle!10, circle, draw=thistle, inner sep=1pt,minimum width=15pt] {12};
\node (w4) at (3*360/7:5cm) [ball color=thistle!10, circle, draw=thistle, inner sep=1pt,minimum width=15pt] {13};
\node (w5) at (4*360/7:5cm) [ball color=thistle!10, circle, draw=thistle, inner sep=1pt,minimum width=15pt] {9};
\node (w6) at (5*360/7:5cm) [ball color=thistle!10, circle, draw=thistle, inner sep=1pt,minimum width=15pt] {8};
\node (w7) at (6*360/7:5cm) [ball color=thistle!10, circle, draw=thistle, inner sep=1pt,minimum width=15pt] {7};  
\end{scope}
}
\end{tikzpicture}
\caption{Example \ref{ex:pls1}}
\label{tikz}

\end{figure}

\begin{example}\label{ex:pls2}
{\rm
{\bf [A properly innately transitive example]} The point set of this
partial linear space consists of the $15$ points of the projective $3$-space $\PG(3,2)$, and the  line set is the set of all lines of $\PG(3,2)$ except for $5$ pairwise disjoint lines (with natural incidence). The automorphism group of this design is $(C_3\times A_5). 2=\GammaL(2,4)$. It is properly innately transitive of rank $3$, and its largest quasiprimitive subgroup is $A_5.2=S_5$ which has rank $4$ on points.  
}
\end{example}

\subsection*{Properly innately transitive groups and special $2$-transitive groups}

Let $\PIT$ denote the family of finite, properly innately transitive permutation groups, and let $\Q$ be the family of finite quasiprimitive permutation groups. Let $G\in\PIT$. Then $G\leq\Sym(\Omega)$ for some finite set $\Omega$ and $G$ has a plinth $M$ (since $G$ is innately transitive). 
Since $G$ is not quasiprimitive, it follows from \cite[Lemma 4.2(2) and Proposition 5.3]{BP} that the centraliser
\begin{equation}\label{e:cent}
    C:=\C_G(M)\ \mbox{satisfies $1\ne C\lhd G$ and $C$ is semiregular and intransitive on $\Omega$.}
\end{equation}
(A permutation group is \emph{semiregular} if the only element fixing a point is the identity.) Moreover $M$ is nonabelian as otherwise $M\subseteq C$ forcing $C$ to be transitive; and $M$ is the unique plinth of $G$ by \cite[Lemma 5.1]{BP}. Thus
\begin{equation}\label{e:plinth}
    M=T^k\ \mbox{for some nonabelian finite simple group  $T$ and integer $k\geq1$.}
\end{equation}
For a group $H$ acting on a set $\Delta$, the permutation group induced by $H$ on $\Delta$ is  denoted by $H^\Delta$. We often have a single group acting on more than one set, and this notation helps us, and we hope also the reader, to be clear about the actions.
This convention means in particular that, since $G\leq {\rm Sym}(\Omega)$, we have $G^\Omega =G$. Also, noting that the set $\Sigma$ of $C$-orbits in $\Omega$ is a $G$-invariant partition of $\Omega$, it follows from \cite[Lemma 4.3 and Proposition 5.3]{BP} that the permutation group induced by $G$ on $\Sigma$ is $G^\Sigma=G/C$ and is quasiprimitive with unique plinth $M^\Sigma = MC/C\cong M$. Further, for $\sigma\in\Sigma$ and a point $\alpha\in\sigma$, it follows from \cite[Lemma 5.4]{BP} that
\begin{equation}\label{e:sigmaaction}
    \mbox{$M_\alpha\lhd M_\sigma$ and the induced group $(M_\sigma)^\sigma = M_\sigma/M_\alpha\cong C$ and is regular on $\sigma$.}
\end{equation}
Thus we have a well-defined map: 
\begin{equation}\label{e:phi}
    \varphi:\PIT\to\Q\quad \mbox{given by  $\varphi(G^\Omega)=G^\Sigma$, for $G\in\PIT$ acting on $\Omega$, and $\Sigma$ as above.}
\end{equation}
The map is not onto $\Q$: indeed, each group $\varphi(G^\Omega)$ has a unique nonabelian plinth but there are quasiprimitive groups with two plinths, or with an elementary abelian plinth, see \cite{P93}. The image $\varphi(G^\Omega)$ may be primitive, as we see from Examples~\ref{ex:pls1} and~\ref{ex:pls2} where these groups are $\PSL(3,2)$ and $S_5$ of degrees $7$ and $5$, respectively. In addition, for each $G^\Sigma=\varphi(G^\Omega)$,  its subgroup $(M_\sigma)^\Sigma$ has a quotient 
$(M_\sigma)^\Sigma/(M_\alpha)^\Sigma\cong M_\sigma/M_\alpha\cong C$, and the stabiliser $(G_\sigma)^\Sigma$ acts by conjugation on this quotient. (This is true because $(M_\sigma)^\Sigma=M_\sigma C/C$ and $(M_\alpha)^\Sigma=M_\alpha C/C$ are both $(G_\sigma)^\Sigma$-invariant, and we note that $(M_\sigma C)/(M_\alpha C)\cong M_\sigma/M_\alpha$ which is isomorphic to $C$ by \eqref{e:sigmaaction}.)
The quotient $M_\sigma/M_\alpha$, or equivalently the group $C$, is not in general uniquely determined by the quasiprimitive group $G^\Sigma$, see Example~\ref{ex:q1} for a small example. 

\begin{example}\label{ex:q1}
{\rm
{\bf [Example to show that $\varphi$ is not one-to-one]} Let $M=\PSL(2,25)$, acting $2$-transitively on the $26$ points of the projective line $\Sigma=\PG(1,25)$, and note that, for  $\sigma\in\Sigma$, $M_\sigma=[C_{5}^2]\rtimes C_{12}$, and for  each $i\in\{2, 3\}$, $M_\sigma$ has a proper normal subgroup $R_i= [C_{5}^2]\rtimes C_{12/i}$ with cyclic quotient $M_\sigma/R_i\cong C_i$.  The action of $M$ by right multiplication  on $\Omega_i=\{R_ix\mid x\in M\}$ is transitive and $\C_{\Sym(\Omega_i)}(M)\cong C_i$. This implies that $G_i:=\C_{\Sym(\Omega_i)}(M)\times M=C_i\times \PSL(2,25) < \Sym(\Omega_i)$ is properly innately transitive with plinth $M$ and hence $G_i^{\Omega_i}$ lies in $\PIT$. Moreover, $\varphi(G_i^{\Omega_i})=M^\Sigma$.
}
\end{example}

We therefore consider the enriched map $\widehat\varphi$ which records both the quasiprimitive group $G^\Sigma$ and also the subgroup $(M_\alpha)^\Sigma$:  
\begin{equation}\label{e:phi2}
    \widehat\varphi:G^\Omega\to (G^\Sigma, R)\ \mbox{with $R=(M_\alpha)^\Sigma$ a proper normal $G_\sigma^\Sigma$-invariant subgroup of $M_\sigma^\Sigma$,}
\end{equation}
and note that, in Example \ref{ex:q1}, we have $\widehat\varphi(G_i)=(M^\Sigma, R_i)$ for $i\in\{2, 3\}$.  In many cases the subgroup $R$ determines the point $\sigma\in\Sigma$ uniquely, for example, this is true if $G^\Sigma$ is primitive and $M^\Omega$ is not regular (Lemma~\ref{lem:uniquesigma}). For this reason we have not incorporated the class $\sigma$ explicitly in our notation for the map $\widehat\varphi$. On the other hand, the map $\widehat\varphi$ does depend on the class $\sigma\in\Sigma$, and $\widehat\varphi$ could have been defined for any $\sigma$. It is useful (especially when studying the preimages of a pair $(G^\Sigma,R)$ as in \eqref{e:phi2}) to introduce the following extension of the notion of permutational isomorphism  for these pairs.

\begin{definition}\label{def:equiv}
{\rm 
Let $(G^\Sigma, R)$ and $(H^\Delta, S)$ be pairs such that $G^\Sigma, H^\Delta\in\mathcal{Q}$ with plinths $M^\Sigma, N^\Delta$ as in \eqref{e:plinth}, and such that, for some $\sigma\in\Sigma$ and $\delta\in\Delta$, the subgroups $R, S$ are proper normal subgroups of $(M^\Sigma)_\sigma, (N^\Delta)_\delta$ which are invariant under $(G^\Sigma)_\sigma, (H^\Delta)_\delta$, respectively.  Then $(G^\Sigma, R)$ and $(H^\Delta, S)$ are said to be \emph{equivalent}, denoted $(G^\Sigma, R)\approx (H^\Delta, S)$, if and only if there exists a group isomorphism $f:G^\Sigma\to H^\Delta$ and a bijection $\psi:\Sigma\to\Delta$ such that $\sigma^\psi = \delta$, $R^f=S$, and $(f,\psi)$ is a \emph{permutational isomorphism}, that is to say, for all $\alpha\in\Sigma, g\in G^\Sigma$ we have $(\alpha^g)^\psi = (\alpha^\psi)^{(g)f}$.
}
\end{definition}

 We note that $\approx$ as defined in Definition~\ref{def:equiv} is an equivalence relation, and in particular $\approx$ is a transitive relation.
A natural example of equivalence occurs for any $(G^\Sigma,R)$ as in \eqref{e:phi2} associated with changing the element $\sigma$: for any  $\tau\in\Sigma$, we may choose $x\in G^\Sigma$ such that $\sigma^x=\tau$. Then taking $\psi=x$,  and $f$ to be  the inner automorphism $\iota_x$ of $G^\Sigma$ corresponding to conjugation by $x$, we find that $(G^\Sigma,R)$ is equivalent to $(G^\Sigma,R^x)$ with $R^x$ a proper  normal $(G^\Sigma)_\tau$-invariant subgroup of $(M^\Sigma)_\tau$.

Our particular interest is the subset $\PIT_3$ of $\PIT$ consisting of all finite, rank $3$, properly innately transitive permutation groups. We show (Lemma~\ref{lem2}) that for each group $G\in\PIT_3$ with $G$ acting on a set $\Omega$, the image $\widehat\varphi(G^\Omega)=(G^\Sigma, R)$ is such that 
\begin{itemize}
    \item $G^\Sigma$ is an almost simple $2$-transitive group with nonabelian simple socle $M^\Sigma\cong M$;
    
    \item $1\ne R\lhd (M^\Sigma)_\sigma$  for some $\sigma\in\Sigma$, and $R$ is $(G^\Sigma)_\sigma$-invariant;
    
    \item $(M^\Sigma)_\sigma/R$ is elementary abelian and nontrivial (and isomorphic to $C$);  
    
    \item  $(G^\Sigma)_\sigma$ acts transitively by conjugation on the  nontrivial elements of $(M^\Sigma)_\sigma/R$. 
\end{itemize}
We call permutation groups $G^\Sigma$ with specified subgroup $R$ having these properties \emph{special pairs}, see Definition~\ref{def1}. 
Our first objective is:

\medskip\noindent
{\bf Objective 1:}\quad Determine each $2$-transitive group  $G^\Sigma$, and each of its subgroups $R$ such that $(G^\Sigma,R)$ is a special pair.

\medskip
This classification is completed in Theorem~\ref{t:special}, where we identify several infinite families and a number of sporadic examples (see Table~\ref{t:allr}). We therefore have an explicit set $\widehat\SS$ of all special pairs, up to equivalence, and we note that the subset $\SS$ consisting of all groups $X  \le {\rm Sym}(\Sigma)$ occurring in special pairs $(X^\Sigma,R)$ certainly contains  $\varphi(\PIT_3)$, but the inclusion may be proper.

Our major challenge is to find preimages in $\PIT$ of special pairs, and among them to identify those of rank $3$.
Let $(X^\Sigma,R)\in\widehat\SS$, where $X$ is a $2$-transitive group on $\Sigma$ with nonabelian socle $M$, and $M_\sigma/R\cong C$ is elementary abelian, for some $\sigma\in\Sigma$. 
We show in Lemma~\ref{lem3}$(a)$ that, for $\Omega=\{Rx\mid x\in M\}$ with $M$ acting by right multiplication, the group $\C_{\Sym(\Omega)}(M)\times M\cong C\times M$, and $(C\times M)^\Omega$ lies in $\PIT$ with plinth $M^\Omega$. Moreover by Lemma~\ref{lem3}$(b)$,  given $G_0^{\Omega_0}\in\PIT$ such that 
$\widehat{\varphi}(G_0^{\Omega_0})\approx (X^\Sigma,R)$, the group $G_0^{\Omega_0}$ is permutationally isomorphic to a group $G^\Omega$ such that $C\times M\leq G\leq N_{\Sym(\Omega)}(M)$ and $G/C\cong X$.

\medskip\noindent
{\bf Objective 2:}\quad Determine which of the special pairs $(X^\Sigma,R)$ lie in $\widehat\varphi(\PIT)$, and for each such pair  find, up to permutational isomorphism, the groups $G^\Omega\in\PIT$ such that  $\widehat\varphi(G^\Omega)=(X^\Sigma,R)$.

\medskip
In  Theorem \ref{t:PITSP} we classify all such special pairs $(X^\Sigma,R)$  up to equivalence, and all groups $G^\Omega$ in $\PIT$ up to permutational isomorphism.
While each of the groups $G^\Omega$ found is an interesting permutation group,  not all of them have rank $3$. Here is a small example.

\begin{example}\label{ex:ch2}
{\rm 
Let $G=C\times M=C_2\times \PSL(2,5)$. Let $\Sigma$ be the set of six points of the projective line ${\rm PG}(1,5)$, admitting the natural $M$-action. Let $\sigma\in\Sigma$ and let $M_\sigma=D_{10}$, the stabiliser of $\sigma$. Let $L$ be the unique subdirect subgroup of $C\times M_{\sigma}$ such that $L\cong D_{10}$, and consider the coset action of $G$ on $\Omega=\{Lx\mid x\in G\}$. Then $M^\Omega$ is a transitive minimal normal subgroup of $G^\Omega\cong G$, and the stabiliser in $M$ of the point $\alpha=L\in\Omega$ is $M_\alpha= M\cap L \cong C_5$, of index $2$ in $M_\sigma$. Thus  $G^\Omega\in\PIT$ and it has rank $4$. The normal subgroup $C^\Omega$ has a set of six orbits of length $2$, and the $G$-action induced on them is equivalent to $G^\Sigma = G/C = (C\times M)/C = M^\Sigma\cong M$, and $R:=(M_\alpha)^\Sigma$ satisfies $C_5=R \lhd (M^\Sigma)_\sigma =D_{10}$. Thus $\widehat\varphi(G^\Omega)=(G^\Sigma,R)$, and this is a special pair. 
}
\end{example}

Thus our final objective is the following, and we achieve it in  Theorem \ref{t:PIT3}.

 \medskip\noindent
{\bf Objective 3:}\quad Determine the family $\PIT_3$. 

\medskip 
Finally we note in passing that, for every group $G^\Omega\in\PIT$ such that $\widehat{\varphi}(G^\Omega)=(X^\Sigma, R)$ is a special pair, each part $\sigma\in\Sigma$ has the property that $(G_\sigma)^\sigma$ is an affine $2$-transitive group (Definition~\ref{def1}). Also, if $|\sigma|\geq 5$, then the induced group $(G_\sigma)^\sigma$ is a proper subgroup of $\Sym(\sigma)$, and hence $\sigma$ is a \emph{beautiful subset} as defined in \cite[Section 1]{GS} and the group $G^\Omega$ is not \emph{binary} in the sense of Cherlin, see \cite[Lemma 2.2]{GS}. Groups with this property are determined in Theorem~\ref{t:PITSP} and correspond to Lines 2 and 4 of Table~\ref{t:allr}. They give examples of non-binary permutation groups, beyond the well-studied primitive ones.

\subsection{Statements of the main results}\label{sub:outline}

Our first result gives a complete classification of special pairs, which were  introduced before Objective 1, and are defined formally in Definition~\ref{def1}. 

\begin{maintheorem}\label{t:special}
Let $X\leq {\rm Sym}(\Sigma)$ be almost simple and $2$-transitive with socle $M$. Then $X$ has a subgroup $R$ such that $(X^\Sigma,R)$ is a special pair, relative to an elementary abelian quotient of order $r$, if and only if $M$, $n=|\Sigma|$, and $r$ are as in one of the lines of Table~$\ref{t:allr}$,  with some extra conditions on $X$ given in the last column. Moreover, in each case $r$ is a prime, and for each $X$ and $r$, the subgroup $R$ is uniquely determined up to conjugation. 
\end{maintheorem}

\medskip\noindent

Concerning the conditions in Lines $2$ and $4$ of Table \ref{t:allr}, we note in particular that if $r=2$ in these Lines, then all the conditions in Table \ref{t:allr} are vacuously satisfied, except for the divisibility condition on $r$ which in these cases forces $q=q_0^a$ to be odd. 
On the other hand, if $r$ is odd then the conditions imply that $q$ is not a prime unless we are in Line $4$ with $r=3$ and $q\equiv -1\pmod{9}$. More details of this implication are given in Remark~\ref{re:Table1}.

\begin{table}[h]
\resizebox{\columnwidth}{!}{
\begin{tabular}{|l|c p{2.2cm} p{5.5cm}|p{5cm}|}
\hline
Line  & $M$                                  & $n$              & $r$    & Conditions on $X$              \\ \hline

1 & $A_5$                                  & $5$        & $3$     & $X=S_5$ \\  
2 & $\PSL(d,q)$                                  & $\frac{q^d-1}{q-1}$ & \begin{minipage}[t]{5cm}prime  such that $r \mid\frac{q - 1}{(d,q - 1)}$\\  and $o_r(q_0)=r-1$ 
\end{minipage}  &   \begin{minipage}[t]{5cm} $|X/(X\cap\PGL(d,q))|=a/j$ \\such that  $(j,r-1)=1$,\\
and $(r-1)\mid (a/j)$\end{minipage} \\ 
3&$\PSL(3,2)$&$7$&$2$&\\

4 & $\PSU(3,q)$                                  & $q^3+1$       &\begin{minipage}[t]{5cm} prime such that  $r \mid \frac{q^2 - 1}{(3,q+1)}$, \\   $o_r(q_0)=r-1$ 
\end{minipage}&       \begin{minipage}[t]{5cm} 
$|X/(X\cap\PGU(3,q))|=2a/j$\\ such that $(j,r-1)=1$,\\
and $(r-1)\mid (2a/j)$\end{minipage}\\ 
5& $\Sp(2d,2)$                                  & $2^{2d-1}+\varepsilon 2^{d-1}$&$2$ &  $\varepsilon=\pm$, $d\geq3$         \\ 
6 & ${\rm Ree}(q)$                              & $q^3+1$ &$2$&  $q=3^{2a+1}> 3$         \\ 
7&$\PSL(2,8)$&$28$&$2$&$X=\PGammaL(2,8)={\rm Ree}(3)$\\
8 & $\rm{M}_{11}$                                    &$11$   &$2$    &          \\          
9 & $\rm{HS}$                                    &$176$     &$2$  &                                          \\  
10 & $\rm{Co}_3$                                    &$276$   &$2$    &                                          \\ 
\hline

\end{tabular}}
\caption{Groups with a special pair; in Lines 2 and 4, $q=q_0^a$ with $q_0$ prime }\label{t:allr}
\end{table}

The next result determines all properly innately transitive groups which map under $\widehat{\varphi}$ to one of the special pairs determined in Theorem~\ref{t:special}. We note that the ranks of the groups in lines 2 and 4 in Table~\ref{t:obj2} depend on arithmetical conditions on $q$ and $r$, (see Table~\ref{t:allr}).

\begin{maintheorem}\label{t:PITSP}
Let $i$ be an integer such that $1\leq i\leq 10$,  and let $X \le {\rm Sym}(\Sigma)$ such that $(X^\Sigma, R)$ is a special pair with $M=\Soc(X)$, $n=|\Sigma|$ and  $r=|M_\sigma:R|$ satisfying the conditions of Line $i$ in Table~$\ref{t:allr}$.  Then there exists a group $G \le {\rm Sym}(\Omega)$ such that  $G^\Omega\in\PIT$ with plinth $M$,   $\widehat{\varphi}(G^\Omega)=(X^\Sigma,R)$ and, up to permutational isomorphism, each such $G^\Omega$ satisfies the following, where $C=\C_{\Sym(\Omega)}(M)$. 

\begin{enumerate}
     \item [$(a)$]
     $i\in\{1,7\}$ and
     $G=N_{\Sym(\Omega)}(M)$ acting on $\Omega$  as in Construction $\ref{con:gen}$, with
     $G=(C_3\times A_5).2$ or $G=C_2\times {\rm Ree}(3)$, of rank $3$ or $4$, respectively.
     
    \item [$(b)$] 
    $i\in\{ 3, 8, 9, 10\}$ or $i=5$, and $G=C\times M$  acting on $\Omega$ as in Construction $\ref{con:gen}$, with $M=\PSL(3,2), \rm{M}_{11}, \rm{HS},\rm{Co}_3,$ or $\Sp(2d,2)$, with $G$ of rank $3, 3, 4, 4, 4$, respectively.  
    
    \item [$(c)$]
    $i\in\{ 2, 4, 6 \}$ and $G$ satisfies $C\times M\leq G\leq \overline{G}< \Sym(\Omega)$ and $G/C\cong X$
    with $i$, $M$, $\overline{G}$, and the rank of ${G}$, as in Table~$\ref{t:obj2}$.
    
\end{enumerate}
\end{maintheorem}

\begin{table}[h]
\begin{tabular}{|l|cccl|}
\hline
$i$  & $M$                                    & $\ff$ as in  & Rank of $G$ & Ref. for rank            \\ \hline

2 & $\PSL(d,q)$                             & Construction~\ref{con:psl}& $3$ or $4$ & Lemma~\ref{lem:sppirSL} \\

4 & $\PSU(3,q)$                             & Construction~\ref{con:psu}&$3$ or $4$ & Lemma~\ref{lem:sppirSU} \\ 

6 & ${\rm Ree}(q)$                          & Construction~\ref{con:ree} &$4$ & Lemma~\ref{lem:reerank}\\ 
\hline

\end{tabular}
\caption{Infinite families of groups for Theorem~\ref{t:PITSP}$(c)$ }\label{t:obj2}
\end{table}

Finally, our third result is the classification of finite properly innately transitive permutation groups of rank $3$.

\begin{maintheorem}\label{t:PIT3}
Let  $G \le {\rm Sym}(\Omega)$ such that $G^\Omega\in\PIT$ with plinth $M$ and $C=\C_G(M)$, and let $\Sigma$ be the set of $C$-orbits in $\Omega$, each of length $r$.
Then $G^{\Omega} \in\PIT_3$ if and only if $M$,  $|\Sigma|$, $r$, and   $G$  are as in Table $\ref{t:allrk3}$.

\end{maintheorem}

We note that each of the first three lines of Table $\ref{t:allrk3}$ yields infinitely many groups in $\PIT_3$, see Remark~\ref{rem:infty}.

\begin{table}[h]
\resizebox{\columnwidth}{!}{
\begin{tabular}{|c| p{1.2cm} | p{4.5cm}|p{7cm}|p{5.4cm}|}
\hline
 $M$                                  & $|\Sigma|$              & $r$    &  $G$    &Conditions on $G$          \\ \hline

 $\PSL(d,q)$                                  & $\frac{q^d-1}{q-1}$ & \begin{minipage}[t]{5cm}prime  such that $r |\frac{q - 1}{(d,q - 1)}$,\\ $o_r(q_0)=r-1$ and\\  $(d,r) \ne (2,2)$
\end{minipage}  & $C_r\times \PSL(d,q)\leq G\leq \GammaL(d,q)/\langle\omega^r I\rangle$ &  \begin{minipage}[t]{5.4cm} 
  $|G^\Sigma/(G^\Sigma\cap\PGL(d,q))|=a/j$ \\with $(j,r-1)=1$\end{minipage} \\ 
$\PSL(2,q)$                                  & $q+1$ & $2$   & $C_2\times \PSL(2,q)\leq G\leq \GammaL(2,q)/\langle\omega^2 I\rangle$ &  \begin{minipage}[t]{5.4cm} $q\equiv 1\pmod 4$\\  $G^\Sigma\not\leq \PSigmaL(2,q)$\end{minipage} \\ 
$\PSU(3,q)$                                  & $q^3+1$       &\begin{minipage}[t]{5cm}odd prime such that   \\ $r | q-1$,  $o_r(q_0)=r-1$ 
\end{minipage}& $C_r\times \PSU(d,q)\leq G\leq \GammaU(d,q)/\langle\omega^r I\rangle$   &   \begin{minipage}[t]{5.4cm} $|G^\Sigma/(G^\Sigma\cap\PGU(3,q))|=2a/j$\\ with $(j,r-1)=1$\end{minipage}\\ 

$A_5$                                  & $5$        & $3$     & $\GammaL(2,4)$&  \\  
$\PSL(3,2)$&$7$&$2$&$C_2\times \PSL(3,2)$ &\\

 $\rm{M}_{11}$                                    &$11$   &$2$    &$C_2\times \rm{M}_{11}$ &         \\          
\hline

\end{tabular}}
\caption{$G\in\PIT_3$;  here $q=q_0^a$ with $q_0$ prime}\label{t:allrk3}
\end{table}

We study general properties of properly innately transitive groups in Section \ref{sect:dihedral}, then we classify special pairs in Section \ref{sect:specialpairs},  giving a proof of Theorem \ref{t:special}. In Sections \ref{s:families} and \ref{sect:spree} we construct some infinite families of groups in $\PIT$. Finally, we prove Theorems \ref{t:PITSP} and \ref{t:PIT3} in Section \ref{sect:rank3}.

\section{Finite properly innately transitive permutation groups}\label{sect:dihedral} 

In this section we prove several fundamental results about finite innately transitive permutation groups. 

\begin{notation}\label{not1}
{\rm
Let $\Om$ be a finite set and $G$ be a properly innately transitive permutation group 
on $\Omega$ with plinth $M$, so $G^\Omega\in\PIT$. By \eqref{e:cent} and \eqref{e:plinth}, $C:=\C_G(M)\ne 1$ and $M=T^k$ for some nonabelian simple group $T$ and $k\geq1$. In particular $C\cap M=1$ so 
\begin{equation}\label{e:H}
    H:=\la C, M\ra \cong C\times M.
\end{equation}
Let $\Si$ denote the set of $C$-orbits in $\Om$ (a $G$-invariant partition of $\Omega$), and let 
$\sigma\in\Si$ and $\al\in \sigma$. By \eqref{e:sigmaaction}, $M_\alpha\lhd M_\sigma$ and $M_\sigma/M_\alpha\cong C$. Recall the definition of $\varphi(G)=G^\Sigma$ and $\widehat \varphi (G)=(G^\Sigma, (M_\alpha)^\Sigma)$ from \eqref{e:phi} and \eqref{e:phi2}. 
}
\end{notation}

Our first result establishes the uniqueness claims for $\sigma$ referred to in the first section.

\begin{lemma}\label{lem:uniquesigma}
Using Notation~\ref{not1}, $(M_\sigma)^\sigma$ is transitive, $G_\sigma=M_\sigma G_\alpha$, and $M_\alpha$ is a $G_\sigma$-invariant proper  normal subgroup of $M_\sigma$. Suppose in addition that $G^\Sigma$ is primitive. Then $\sigma$ is the unique element of $\Sigma$ fixed by $M_\alpha$ if and only if $M^\Omega$ is not regular,  and in this case $C=\C_{\Sym(\Omega)}(M)$. 
\end{lemma}

\begin{proof}
 We note  that $(M_\sigma)^\sigma$ is transitive since $M$ is transitive on $\Om$, and hence  $G_\sigma=M_\sigma G_\alpha$.  By \eqref{e:sigmaaction}, $M_\alpha$ is a proper normal subgroup of $M_\sigma$, and as $M_\alpha$ is also normal in $G_\alpha$, $N_G(M_\alpha)$ contains $G_\sigma =  M_\sigma G_\alpha$. This proves the first assertion.
 
 Suppose now that $G^\Sigma$ is primitive, so $G_\sigma$ is a maximal subgroup of $G$. Then $N_G(M_\alpha)$ is equal to either $G$ or $G_\sigma$.  If $M^\Omega$ is regular then $M_\alpha=1$ and so $M_\alpha\lhd G$ and $M_\alpha$ leaves invariant each element of $\Sigma$. Suppose to the contrary that $M^\Omega$ is not regular, so $M_\alpha\ne 1$.  If $N_G(M_\alpha)=G$ then $M_\alpha$ would fix each point of $\Omega$ which is a contradiction, and hence we conclude that $N_G(M_\alpha)=G_\sigma$. It follows from \cite[Lemma 2.19]{PS} that $\sigma$ is the unique element of $\Sigma$ fixed by $M_\alpha$.  Finally, in this case, we have shown that $N_G(M_\alpha)=G_\sigma$ and hence $N_M(M_\alpha)=M_\sigma$. By definition $C \leq \C_{\Sym(\Omega)}(M)$, and by \cite[Theorem 3.2(i)]{PS}, $\C_{\Sym(\Omega)}(M)\cong N_M(M_\alpha)/M_\alpha = M_\sigma/M_\alpha$, and as in Notation~\ref{not1}, $M_\sigma/M_\alpha\cong C$. It follows that $C=\C_{\Sym(\Omega)}(M)$.
\end{proof}

\subsection{Special pairs from rank $3$ properly innately transitive groups}

In our first result we derive additional properties of $C$ and $G^\Sigma$ in the case where $G^\Omega$ has rank $3$. 

\begin{lemma}\label{lem1}
Using Notation~\ref{not1}, assume also that $G^\Omega\in\PIT_3$. 
The following hold:
\begin{enumerate}
\item[$(a)$] $M_\al$ is a $G_\sigma$-invariant proper nontrivial normal subgroup of $M_\sigma$ and  $M_\sigma/M_\alpha\cong C=C_p^c$, for some prime $p$ and positive integer $c$;

\item[$(b)$] $(G_\sigma)^\sigma$ is an affine $2$-transitive subgroup of $\AGL(c,p)$ with group of translations $(H_\sigma)^\sigma =(M_\sigma)^\sigma =C^\sigma \cong C$;
\item[$(c)$] there exists an isomorphism $\psi:C\rightarrow M_\sigma/M_\alpha$ such that 
\[
H_\al = \{ ( a, b) \mid a\in C,\ b\in M_\sigma,\ (a)\psi = bM_\al \} < C\times M_\sigma=H_\sigma;
\]
\item[$(d)$] $G^\Si \cong G/C$ is an almost simple $2$-transitive group with nonabelian simple socle 
$H^\Si \cong M^\Si\cong M$, and  $G_\al \cong G_\sigma/C$ acts transitively by conjugation on the nontrivial elements of the quotient
$M_\sigma/M_\al\cong H_\sigma/H_\al = (H_\sigma)^\sigma \cong C_p^c$. In particular $|G:H|=|G_\alpha:H_\alpha|$ is divisible by $p^c-1$.
\end{enumerate}
\end{lemma}

\begin{proof}
$(a)$ By Lemma~\ref{lem:uniquesigma}, $M_\al$ is a $G_\sigma$-invariant proper normal subgroup of $M_\sigma$. By \eqref{e:cent}, $C$ is semiregular on $\Om$.  In particular, $C^\sigma\cong C$ is regular.
Since $G$ has rank $3$ on $\Om$ and $\sigma$ is a block of imprimitivity for $G$ in $\Omega$, the $G_\al$-orbits are $\{\al\}, \sigma\setminus \{\al\}, \Om\setminus \sigma$. This means firstly that $(G_\sigma)^\sigma$ is $2$-transitive with a regular normal subgroup $C^\sigma$, and it follows that $C^\sigma\cong C$ is elementary abelian (see \cite[Lemma 3.19]{PS}), so $M_\sigma/M_\alpha\cong C=C_p^c$, for some prime $p$ and positive integer $c$. Thus  $(G_\sigma)^\sigma$ is an affine $2$-transitive subgroup of $\AGL(c,p)$. Secondly the transitivity of $G_\alpha$ on $ \Om\setminus \sigma$ implies that $G_\sigma$ is transitive on $\Sigma\setminus\{\sigma\}$ and hence $G^\Sigma$ is $2$-transitive. Then it follows from the discussion in Section~\ref{sect:intro}, especially \eqref{e:plinth}, that $M^\Sigma\cong M$ is the socle  of the $2$-transitive group $G^\Sigma$ and $M$ is nonabelian. Hence $M$ is a nonabelian simple group by a  theorem of Burnside (see \cite[Exercise 12.4]{WIE} or \cite[Theorem 3.21]{PS}). To complete the proof of part (a) we need to show that $M_\alpha$ is nontrivial.  Suppose to the contrary that $M_\alpha=1$. Then $M^\Sigma$ is the socle of the almost simple $2$-transitive permutation group $G^\Sigma$, with abelian point stabiliser $(M^\Sigma)_\sigma\cong M_\sigma \cong M_\sigma/M_\alpha\cong C_p^c$. 
Note that $M^\Sigma\cong M$ is primitive, and hence $M_\sigma$ is a maximal subgroup of $M$, see \cite[Corollary 2.15 and Proposition 3.20]{PS}. Let $\sigma'\in\Sigma\setminus\{\sigma\}$, and let $L = M_{\sigma,\sigma'}$. The subgroup $M_\sigma$ does not fix $\sigma'$ by \cite[Corollary 2.21]{PS}, and hence $M_\sigma\ne M_{\sigma'}$. In particular, $\langle M_{\sigma}, M_{\sigma'}\rangle = M$ since $M_\sigma$ is maximal in $M$. On the other hand each of the abelian subgroups $M_\sigma$ and $M_{\sigma'}$ normalises $L$, and hence $L$ is normalised by $\langle M_{\sigma}, M_{\sigma'}\rangle = M$. Since $M$ acts faithfully on $\Sigma$ this implies that $L=1$, and since this holds for all $\sigma'\ne \sigma$ it follows that $M$ is a Frobenius group. Thus $M$ has a nontrivial nilpotent normal subgroup, by \cite[Theorem 3.16]{PS}, and this is a contradiction since $M$ is a nonabelian simple group.
Thus $M_\alpha\ne1$ and part (a) is proved. 

$(b)$  By the previous paragraph,  $H_\sigma=C\times M_\sigma$, and so $(M_\sigma)^\sigma$ centralises $C^\sigma$.  Since $\C^\sigma$ is abelian and transitive, it follows from \cite[Theorem 3.2]{PS} that $(H_\sigma)^\sigma =(M_\sigma)^\sigma =C^\sigma $, and part $(b)$ is proved.  

$(c)$ Since each of $C$ and $M_\sigma$ is transitive on $\sigma$, we have $H_\sigma = CH_\al=M_\sigma H_\al$, and it follows that $H_\al$ is a subdirect subgroup of the direct product $H_\sigma=C\times M_\sigma$. Also $H_\al\cap C=1$ (since $C$ is semiregular) and $H_\al\cap M_\sigma = M_\al$.  
By Goursat's Lemma \cite[Theorem 4.8]{PS}, there exists an isomorphism $\psi: C\rightarrow  M_\sigma/M_\al$ such that $H_\al = \{ (a, b) \mid a\in C, b\in M_\sigma, (a)\psi = bM_\al \}$, proving part $(c)$.

$(d)$ As $C\lhd G$, $\Si$ is a $G$-invariant partition of $\Omega$, and by \cite[Lemma 4.3 and Proposition 5.3]{BP}, $G^\Sigma\cong G/C$ (the kernel of the $G$-action on $\Si$ is precisely $C$). We showed in  the proof of part (a) that  $M^\Sigma\cong M$ is the nonabelian simple socle of the $2$-transitive group $G^\Sigma$, and we have $H^\Sigma = H/C=MC/C\cong M$. Further, as noted in the previous paragraph,  
$H_\sigma=M_\sigma H_\alpha$. By part $(b)$, $(H_\sigma)^\sigma$ is regular, so the groups $H_\alpha, M_\alpha$ are the kernels of the actions of $H_\sigma, M_\sigma$ on $\sigma$, respectively. Thus $M_\sigma/M_\al\cong H_\sigma/H_\al = (H_\sigma)^\sigma \cong C_p^c$.
Since $C$ is regular and faithful on $\sigma$, we have $G_\sigma=G_\alpha C$ and $G_\alpha\cap C=1$. Hence $G_\sigma/C \cong G_\alpha/(G_\alpha\cap C)\cong G_\alpha$. Finally, $G_\sigma = M_\sigma G_\alpha$ (by Lemma~\ref{lem:uniquesigma}) and the $2$-transitive group $(G_\sigma)^\sigma=G_\sigma/G_{(\sigma)}$ contains $(M_\sigma G_{(\sigma)}/G_{(\sigma)})\cong M_\sigma/M_\alpha$ as a regular normal subgroup. Thus, by \cite[Theorem 11.2]{WIE} we may identify $\sigma$ with the group $M_\sigma/M_\alpha$ in such a way that $M_\sigma/M_\alpha$ acts by right multiplication, $\alpha$ corresponds to the identity element of $M_\sigma/M_\alpha$, and $G_\alpha/G_{(\sigma)}$ acts by conjugation. Since, as we noted above, $G_\alpha$ is transitive on $\sigma\setminus\{\alpha\}$, it follows that the conjugation action induced by $G_\alpha$ on $M_\sigma/M_\alpha$ is transitive on the nontrivial elements of $M_\sigma/M_\alpha$. Now $H_\alpha$ acts trivially in this action (since $C$ centralises $M$, and $M_\sigma/M_\alpha$ is abelian), and hence $p^c-1$ divides $|G_\alpha:H_\alpha|$. Since $H$ is transitive we have $G=HG_\alpha$, and hence $|G_\alpha:H_\alpha|= |G:H|$. This proves part $(d)$. 
 \end{proof} 

The properties discussed in Lemma~\ref{lem1} enable us to  show that a rank $3$ group in $\PIT$ gives rise to a special pair, which we define formally as follows.

\begin{definition}\label{def1}
{\rm 
Let $X\leq {\rm Sym}(\Sigma)$, let $\sigma\in\Sigma$ and let $R\leq X_\sigma$. We say that $(X^\Sigma,R)$ is a \emph{special pair} if the following properties all hold.
\begin{enumerate}
    \item[$(a)$] $X$ is $2$-transitive on $\Sigma$ with a nonabelian simple socle $M^\Sigma$;
    \item[$(b)$] $R$ is a proper, nontrivial, $X_\sigma$-invariant, normal subgroup of the stabiliser $(M^\Sigma)_\sigma$, and $(M^\Sigma)_\sigma/R=C_p^c$, for some prime $p$ and integer $c\geq1$;
    \item[$(c)$] the induced conjugation action of $X_\sigma$ on $(M^\Sigma)_\sigma/R$ is transitive on the $p^c-1$ nontrivial elements of $(M^\Sigma)_\sigma/R$.
\end{enumerate}
}
\end{definition}

Note that condition $(c)$ of Definition~\ref{def1} is trivially satisfied if $p^c=2$.
We use Lemma~\ref{lem1} to show that each group in $\PIT_3$ determines a special pair under the map $\widehat\varphi$ of \eqref{e:phi2}.

\begin{lemma}\label{lem2}
Use the notation of Lemma~$\ref{lem1}$, and assume that $G^\Omega\in\PIT_3$. Let $\sigma\in\Sigma$ and let $R:= CM_\alpha/C = (M_\alpha)^\Sigma$. Then
\begin{enumerate}
    \item[$(a)$] $G^\Sigma$ is an almost simple $2$-transitive group with nonabelian simple socle $H^\Sigma=M^\Sigma\cong M$.
    
    \item[$(b)$] $R$ is a $(G^\Sigma)_\sigma$-invariant proper nontrivial normal subgroup of $(H^\Sigma)_\sigma = (C\times M_\sigma)/C$,  and the induced conjugation action of  $(G^\Sigma)_\sigma$ on $(H^\Sigma)_\sigma/R\cong M_\sigma/M_\alpha\cong C_p^c$ is transitive on the nontrivial elements of $(H^\Sigma)_\sigma/R$;
    
    \item[$(c)$] $(G^\Sigma,R)$ is a special pair, and $\widehat\varphi(G^\Omega) = (G^\Sigma,R)$.
\end{enumerate}
\end{lemma}

\begin{proof}
Part $(a)$ follows from Lemma~\ref{lem1}$(d)$. By definition, $R=CM_\alpha/C$, and it follows from Lemma~\ref{lem1}$(a)$ that $R$ is a $(G^\Sigma)_\sigma$-invariant  proper nontrivial  normal subgroup of $CM_\sigma/C=H_\sigma/C=(H^\Sigma)_\sigma$, and that $M_\sigma/M_\alpha\cong C_p^c$. Now 
\[
(H^\Sigma)_\sigma/R = (H_\sigma/C)/(CM_\alpha/C)
\cong H_\sigma/(CM_\alpha) = (CM_\sigma)/(CM_\alpha)\cong M_\sigma/M_\alpha\cong C_p^c.
\]
Also $G_\sigma = CG_\alpha$ (as $C^\sigma$ is transitive), and so $(G^\Sigma)_\sigma = (G_\alpha)^\Sigma$. Each $R$-coset in $(H^\Sigma)_\sigma=H_\sigma/C$ has a representative of the form $xC$ with $x\in M_\sigma$ (since $H_\sigma = C\times M_\sigma$), and an element $g\in G_\alpha$ conjugates this coset to the $R$-coset with representative $x^gC$. Thus the conjugation action of $G_\alpha$ on $(H^\Sigma)_\sigma/R$ is equivalent to its action on $M_\sigma/M_\alpha$, and by Lemma~\ref{lem1}$(d)$ the latter action is transitive on non-identity elements. Thus part $(b)$ is proved. It follows from Definition~\ref{def1} and part $(a)$ that $(G^\Sigma,R)$ is a special pair, and from \eqref{e:phi2} that $\widehat\varphi(G^\Omega) = (G^\Sigma,R)$.
\end{proof}

We record, for later use, a property of groups $G^\Omega\in\PIT$ which map to special pairs.

\begin{lemma}\label{lem2b}
Using Notation~\ref{not1}, assume that $G^\Omega\in\PIT$ is such that $\widehat\varphi(G^\Omega) = (G^\Sigma,R)$ is a special pair, where $R:= CM_\alpha/C = (M_\alpha)^\Sigma$ and $\alpha\in\sigma\in\Sigma$. Then $G_\sigma^\sigma$ is $2$-transitive.
\end{lemma}

\begin{proof}
 By Definition~\ref{def1}, the induced conjugation action of $G_\sigma$ on $(M^\Sigma)_\sigma/R\cong  M_\sigma/M_\alpha$ is transitive on the $|\sigma|-1$ nontrivial cosets. As explained in the proof of Lemma~\ref{lem1}$(d)$, $(M^\Sigma)_\sigma/R$ is a regular normal subgroup of $(G_\sigma)^\sigma$, and this $G_\sigma$-conjugation action is equivalent to the $G_\sigma$-action on $\sigma\setminus\{\alpha\}$. It follows that $G_\sigma^\sigma$ is 2-transitive.
\end{proof}

\subsection{Properly innately transitive groups from special pairs}

We next study special pairs $(X,R)$ to understand which groups $G^\Omega\in\PIT$ map to them under $\widehat\varphi$, up to equivalence. 
(Recall the notion of equivalence given in Definition~\ref{def:equiv}.) We begin with a basic construction which in almost all cases produces examples of such groups.

\begin{construction}\label{con:gen}
Suppose that $(X^\Sigma,R^\Sigma)$ is a special pair, where $X$ is an almost simple $2$-transitive group on a set $\Sigma$ with nonabelian simple socle $M=\Soc(X)$ and, for some $\sigma\in\Sigma$, $1\ne R\lhd M_\sigma$ with $R$ invariant under $X_\sigma$ and $M_\sigma/R\cong C_p^c$, for some prime $p$ and integer $c$. 
Let $\Omega=\{Rx\mid x\in M\}$ with $M$ acting by right multiplication, and let $C:=\C_{\Sym(\Omega)}(M)$. Then both $N:=N_{\Sym(\Omega)}(M)$, and its subgroup $H=C\times M$ lie in $\PIT$ and  preserve the partition of $\Omega$ formed by the set of $C$-orbits.
\end{construction}

\begin{lemma}\label{lem3}
Using the notation of Construction $\ref{con:gen}$, the group  $C\cong C_p^c$, and 
the following hold.
\begin{enumerate}
    \item[$(a)$] $H\leq N\leq \Sym(\Omega)$, the set of $C$-orbits in $\Omega$ is $N$-invariant,  $H^\Omega, N^\Omega\in\PIT$ both with plinth $M^\Omega \cong M$, and $\widehat\varphi(H^\Omega)\approx (M^\Sigma,R^\Sigma)$.
    \item[$(b)$] Suppose that
       $G_0\leq \Sym(\Omega_0)$ satisfies $G_0^{\Omega_0}\in\PIT$ with plinth $M_0$, and let $N_0:=N_{\Sym(\Omega_0)}(M_0)$, $C_0:=\C_{\Sym(\Omega_0)}(M_0)$,  $\Sigma_0$ be the set of $C_0$-orbits in $\Omega_0$, and $\alpha_0\in\Omega_0$. Then $\widehat{\varphi}(G_0^{\Omega_0})=(G_0^{\Sigma_0},R_0^{\Sigma_0})$ where  $R_0=(M_0)_{\alpha_0}$. Suppose also that  $\widehat{\varphi}(G_0^{\Omega_0})\approx (X^\Sigma,R^\Sigma)$ (with $\approx$ as in Definition~\ref{def:equiv}).  Then  there exists a permutational isomorphism $(f,\psi)$ from $N_0^{\Omega_0}$ to $N^\Omega$ such that 
       \begin{enumerate}
           \item[(i)] $M_0^f=M$, $C_0^f=C$, $\alpha_0^\psi=R\in\Omega$, and  $G:=G_0^f$ satisfies $C\times M\leq G\leq N= N_{\Sym(\Omega)}(M)$; and moreover
    
        \item[(ii)] 
 $G^\Omega$ and $G_0^{\Omega_0}$ are permutationally isomorphic, and  $\widehat{\varphi}(G^{\Omega})\approx (X^\Sigma,R^\Sigma)$.
       \end{enumerate}
\end{enumerate}
\end{lemma}

Note that, by Lemma~\ref{lem:uniquesigma}, $C=\C_G(M)=\C_{\Sym(\Omega)}(M)$ since $R\ne1$.

\begin{proof}
First we prove the assertions about $C$.   By definition $M^\Omega$ is transitive with $R$ the stabiliser of the trivial coset $\alpha=R\in\Omega$. Also, by the definition of a special pair, $N_M(R)$ contains $M_\sigma$, as $R\lhd M_\sigma$. Since $M^\Sigma\cong M$ is the socle of the almost simple $2$-transitive group $X^\Sigma$, it follows from \cite[Corollary 2.15 and Proposition 3.20]{PS} that $M^\Sigma$ 
is primitive and $M_\sigma$ is a maximal subgroup of $M$. Hence  $N_M(R)=M_\sigma$ or $M$. However if $N_M(R)=M$ then $R\lhd M$ and as $R$ fixes the point $\sigma\in\Sigma$ and $M$ is faithful on $\Sigma$, we conclude that $R=1$, contradicting Definition~\ref{def1}. Thus $N_M(R)=M_\sigma$. It follows from \cite[Theorem 3.2(i)]{PS} that $C=\C_{\Sym(\Omega)}(M) \cong M_\sigma/R\cong C_p^c$. 

\medskip
$(a)$ Since $C\cap M \lhd M$, $M$ is a nonabelian simple group and $C$ is abelian, we have  $C\cap M=1$, and hence $H=C\times M$ and $H\leq N=N_{\Sym(\Omega)}(M)$. Since $M$ is a nonabelian simple group it follows that $M^\Omega\cong M$ is a transitive minimal normal subgroup of both $H^\Omega=C\times M$ and $N^\Omega$, and hence $H^\Omega, N^\Omega\in\PIT$ both with plinth $M^\Omega$. 
 Now 
by definition, $C\lhd N$, and so  the set $\Sigma'$ of $C$-orbits in $\Omega$ forms an $N$-invariant partition of $\Omega$. By \cite[Theorem 3.2(ii)]{PS}, the $C$-orbit containing $\alpha:=R\in\Omega$ is equal to $\{ Rx\mid x\in N_M(R)\}$. Since $N_M(R)=M_\sigma$, the setwise stabiliser in $M$ of this $C$-orbit is $M_\sigma$, and hence both $M^{\Sigma'}$ and $M^\Sigma$ are permutationally isomorphic to the action by right multiplication of $M$ on $\{M_\sigma x\mid x\in M\}$. Further, it follows from \cite[Lemma 2.8]{PS} that there exists a bijection $\psi':\Sigma'\to\Sigma$ such that $\alpha^C\to \sigma$ and $({\rm id}, \psi')$ is a permutational isomorphism from $M^{\Sigma'}$ to $M^\Sigma$ which sends $R^{\Sigma'}$ to $R^\Sigma$. Now, since $C$ acts trivially on $\Sigma'$, we have $H^{\Sigma'}=M^{\Sigma'}$, and so by \eqref{e:phi2}, $\widehat\varphi(H^\Omega)= (M^{\Sigma'}, R^{\Sigma'})$, which, by Definition~\ref{def:equiv} is equivalent to $(M^\Sigma,R^\Sigma)$. This proves part $(a)$.

\medskip

$(b)$ 
By \eqref{e:cent},
$C_0$ is a nontrivial intransitive normal subgroup of $G_0^{\Omega_0}$. Let  $\sigma_0\in\Sigma_0$ such that $\alpha_0\in\sigma_0$. By \eqref{e:sigmaaction}, $R_0 =(M_0)_{\alpha_0}\lhd (M_0)_{\sigma_0}$ and $((M_0)_{\sigma_0})^{\sigma_0} = (M_0)_{\sigma_0}/R_0\cong C_0$, and by \cite[Lemma 4.3 and Proposition 5.3]{BP}, $G_0^{\Sigma_0}=G_0/C_0$ is quasiprimitive with socle $M_0^{\Sigma_0}=M_0C_0/C_0\cong M_0$.
Further, by \eqref{e:phi2}, $\widehat{\varphi}(G_0^{\Omega_0})=(G_0^{\Sigma_0},R_0^{\Sigma_0})$ and hence $(G_0^{\Sigma_0},R_0^{\Sigma_0})\approx (X^\Sigma,R^\Sigma)$. Thus by Definition~\ref{def:equiv}, there is an isomorphism $f_0:G_0^{\Sigma_0}\to X$ and a bijection $\psi_0:\Sigma_0\to \Sigma$ such that $\psi_0:\sigma_0\to\sigma$, $f_0:R_0\to R$, and $(f_0,\psi_0)$ is a permutational isomorphism from $G_0^{\Sigma_0}$ to $X^\Sigma$.  Then  $M^\Sigma=\Soc(X)^\Sigma= (M_0^{\Sigma_0})^{f_0}$, and $(M_\sigma)^\Sigma= (((M_0)_{\sigma_0})^{\Sigma_0})^{f_0}$. Hence, identifying $M_0$ with $M_0^{\Sigma_0}$, the restriction $f_0|_{M_0}$ is an isomorphism $M_0\to M$ such that $(M_0)_{\sigma_0}\to M_\sigma$ and $R_0\to R$. 

By \cite[Lemma 2.8]{PS}, the map $\alpha_0^x\to R_0x$ (for $x\in M_0$) is a bijection from $\Omega_0$ to the set $[M_0:R_0]=\{ R_0x\mid x\in M_0\}$ of right cosets of $R_0$ in $M_0$ such that $M_0^{\Omega_0}$ is permutationally isomorphic to the right multiplication action of $M_0$ on $[M_0:R_0]$. We will therefore identify $\Omega_0$ with $[M_0:R_0]$ so that $\alpha_0=R_0$, and we note that $\Omega=[M:R]$ (see  Construction~\ref{con:gen}). Next we define $\psi:\Omega_0\to \Omega$ by $\psi:R_0 x\to R (x^{f_0})$ (for $x\in M_0$). Then $\psi$ is a bijection such that $\alpha_0^\psi=R$ and, for all $x, y\in M_0$,
\[
((R_0 x)^{\psi}){y^{f_0}} =(R x^{f_0})
y^{f_0} = R (x^{f_0}\,{y^{f_0}}) = 
R((xy)^{f_0}) = (R_0(xy))^{\psi} = ((R_0 x){y})^{\psi}
\]
and hence $(f_0|_{M_0}, \psi)$ is a permutational isomorphism from $M_0^{\Omega_0}$ to $M^{\Omega}$. 
We now extend this permutational isomorphism. First we define a map $f:\Sym(\Omega_0)\to\Sym(\Omega)$ such that, for $y\in\Sym(\Omega_0)$, the image $y^f$ is defined by 
\[
\omega^{y^f} := ((\omega^{\psi^{-1}})^y)^\psi, \quad \mbox{for all $\omega\in\Omega$.} 
\]
The map $f$ is a group homomorphism since, for all $x,y\in\Sym(\Omega_0)$ and $\omega\in\Omega$,
\[
\omega^{(xy)^f} = ((\omega^{\psi^{-1}})^{xy})^\psi
=(((\omega^{\psi^{-1}})^x)^{y})^\psi
=(((\omega^{x^f})^{\psi^{-1}})^y)^\psi
=(\omega^{x^f})^{y^f}
\]
and hence $(xy)^f=x^fy^f$. It is tedious but straightforward to check that $f$ is one-to-one, and hence $f$ is a group isomorphism. 

\medskip\noindent{\textit{Claim 1:} $f$ extends the map $f_0|_{M_0}$, that is to say, $f|_{M_0}=f_0|_{M_0}$, and in particular $R_0^f=R$.}

 We prove this using the definition of $\psi$. Each point $\omega\in\Omega$ is of the form $\omega=(R_0 x)^\psi = Rx^{f_0}$ for some $x\in M_0$, and for each $y\in M_0$, the action of $y^f$ on $\omega$ satisfies
\[
\omega^{y^f} = ((\omega^{\psi^{-1}})^y)^\psi
= ((R_0 x)^y)^\psi = (R_0 xy)^\psi = R(xy)^{f_0} = (R x^{f_0})y^{f_0} = \omega^{y^{f_0}}.
\]
Since this holds for all $\omega\in\Omega$ it follows  that $y^{f_0}=y^f$, and, since this holds for all $y\in M_0$, we have $f|_{M_0}=f_0|_{M_0}$. Then $R_0^f=R_0^{f_0}=R$, and  Claim 1 is proved.

\medskip\noindent{\textit{Claim 2:} $f$ maps $N_0=N_{\Sym(\Omega_0)}(M_0)$ to $N=N_{\Sym(\Omega)}(M)$, and $C_0$ to $C$.}

Let $y\in N_0$. Each element of $M$ is of the form $x^{f_0}$ for some unique
element $x\in M_0$. Note that
\[
\begin{array}{ccll}
    (x^{f_0})^{y^f} &=& (x^f)^{y^f} &\mbox{as $f_0=f|_{M_0}$}    \\
    &=& (x^y)^f &\mbox{as $f$ is a homomorphism}    \\
    &=& (x^y)^{f_0} \in M &\mbox{as $x^y\in M_0$ and $f_0=f|_{M_0}$.}    \\
\end{array} 
\]
Since this holds for all $x^{f_0}\in M$ it follows that $y^f\in N$, and hence $N_0^{f}\leq N$. A similar proof yields the reverse inclusion and we conclude that $N_0^f=N$. A similar proof shows that $C_0^f=C$, proving Claim 2.

Thus $(f|_{N_0},\psi)$ is a permutational isomorphism from $N_0^{\Omega_0}$ to $N^\Omega$, and we now view $f$ as an isomorphism $f:N_0\to N$ (that is, we replace $f$ by $f|_{N_0}$). By Claims 1 and 2, $M_0^f=M$ and $C_0^f=C$, and by definition $\alpha_0^\psi = R\in\Omega$. Hence $\psi$ maps $C_0$-orbits in $\Omega_0$ to $C$-orbits in $\Omega$, and in particular $\sigma_0^\psi$ is the $C$-orbit containing $\alpha_0^\psi = R$, that is, $\sigma_0^\psi=\sigma$.   Let $G:= G_0^f$. Then since $G_0\leq N_0$ (since $M_0\lhd G_0$) it follows that $(f|_{G_0},\psi)$ is a permutational isomorphism from $G_0^{\Omega_0}$ to $G^\Omega$, and since $C_0\times M_0\leq G_0\leq N_0$ we have $C\times M\leq G\leq N$. 

Now since $C_0^f=C$, the map $f$ induces a group isomorphism $\overline{f}:G_0/C_0\to G/C$ given by $\overline{f}:C_0 x\to C x^f$. Also since $\psi$ maps $C_0$-orbits in $\Omega_0$ to $C$-orbits in $\Omega$, $\psi$ induces a bijection $\overline{\psi}:\Sigma_0 \to \Sigma'$ given by $\overline{\psi}:\omega^{C_0} \to (\omega^\psi)^C$. (This map is well defined since, for $x\in C_0$, we have from the definition of $f$ that 
\[
(\omega^\psi)^{x^f}=(((\omega^\psi)^{\psi^{-1}})^x)^\psi = (\omega^x)^\psi 
\]
so $\omega^\psi, (\omega^x)^\psi$ lie in the same $C$-orbit.) It is tedious but straightforward to show that $(\overline{f}, \overline{\psi})$ is a permutational isomorphism from $G_0^{\Sigma_0}$ to $G^{\Sigma'}$. Moreover, $\sigma_0^{\overline{\psi}}=\sigma$ and
$(R_0^{\Sigma_0})^{\overline{f}} = (R_0^f)^{\Sigma'}$, and by Claim 1,  $R_0^f=R$. Hence, $(R_0^{\Sigma_0})^{\overline{f}} = R^{\Sigma'}$, and so by Definition~\ref{def:equiv}, 
$\widehat{\varphi}(G_0^{\Omega_0}) = (G_0^{\Sigma_0}, R_0^{\Sigma_0})\approx (G^{\Sigma'}, R^{\Sigma'}) = \widehat{\varphi}(G^{\Omega})$. By assumption, $\widehat{\varphi}(G_0^{\Omega_0})\approx (X^\Sigma,R^\Sigma)$, and hence $\widehat{\varphi}(G^{\Omega}) \approx (X^\Sigma,R^\Sigma)$.
\end{proof}

\begin{corollary}\label{lem:uniq}
Let $G \le \Sym(\Omega)$ such that $G^{\Omega} \in \PIT$, and let $M$, $C$ and $\Sigma$ be as in Notation~\ref{not1}. Also let $R=M_{\alpha}$ for some $\alpha \in \Omega$, and suppose that $(G^{\Sigma},R^{\Sigma})$ is a special pair and $N_{\Sym(\Sigma)}(M^{\Sigma})=M^{\Sigma}$.  Then $N_{\Sym(\Omega)}(M)=G=C\times M$.
\end{corollary}

\begin{proof}

Since $M^{\Sigma} \le G^{\Sigma} \le N_{\Sym(\Sigma)}(M^{\Sigma})=M^{\Sigma}$, $N_{\Sym(\Omega)}(M)/C \cong N_{\Sym(\Omega)}(M)^{\Sigma} \le N_{\Sym(\Sigma)}(M^{\Sigma})$ and $G^{\Sigma}\cong G/C$, it follows that $N_{\Sym(\Omega)}(M)^{\Sigma}= G^\Sigma=M^\Sigma$ and $N_{\Sym(\Omega)}(M)=G=C \times M$, proving the statement.
\end{proof}

\section{Classification of special pairs} \label{sect:specialpairs}

In this section we first prove Theorem~\ref{t:special}, classifying all special pairs as defined in Definition~\ref{def1}, and then deduce several consequences which will prove useful for achieving {\bf Objectives 2} and {\bf 3}.  

\subsection{Proof of Theorem~\ref{t:special}}
Suppose that $X, M, \Sigma, \sigma$ are as in the statement of Theorem~\ref{t:special},  and that $(X^\Sigma,R)$ is a special pair as in Definition~\ref{def1} 
with $R\lhd M_\sigma$. Let $n=|\Sigma|$ and $r= p^c=|M_\sigma/R|$. Then $X^\Sigma$ is a finite almost simple $2$-transitive permutation group, and the proof uses the classification of such groups. We consider each family or sporadic example separately, using the list in \cite[pp. 196--197]{Cam99}.

\smallskip\noindent {\em Case: alternating:}\quad  Here $M=A_n$ with $n\geq 5$. Since $M_\sigma=A_{n-1}$ has a nontrivial elementary abelian quotient it follows that $n=5$, $r=3$, $R$ is unique, and condition Definition~\ref{def1}$(c)$ 
forces $X=S_5$. So Line 1 of Table~$\ref{t:allr}$ holds, and all conditions of Definition~\ref{def1} hold for this line.

\smallskip
\noindent {\em Case:\ linear:}\quad 
Here  $M=\PSL(d,q)$ of degree  $n=(q^d-1)/(q-1)$, with $d\geq 2$ and $(d,q)\ne(2,2)$ or $(2,3)$, and we may take $\Sigma$ to be the set of $1$-spaces of the vector space $V(d,q)=\mathbb{F}_q^d$ of $d$-dimensional row vectors. 

First we determine $r$.
If $d=2$, so that $q\geq4$,  then the stabiliser $M_\sigma=[q] \rtimes C_{(q-1)/(2,q-1)}$, and  an elementary abelian quotient of $M_\sigma$ must be an elementary abelian quotient of $C_{(q-1)/(2,q-1)}$. Thus $r$ is a prime divisor of $(q-1)/(2,q-1)$ and there is a unique $R$ for each such prime $r$.
Suppose now that $d\geq3$. Considering the (possibly unfaithful) action of $\widehat{M}=\SL(d,q)$ on $\Sigma$ we have 
$\widehat{M}_\sigma=[q^{d-1}] \rtimes \GL(d-1,q)$, which contains a subgroup $Z_0$ of scalars of order $(d,q-1)$, and $M_\sigma = \widehat{M}_\sigma/Z_0$. If $\PSL(d-1,q)$ is a nonabelian simple group then the largest abelian quotient of $M_\sigma$ is cyclic of order $(q-1)/(d,q-1)$, and so $r$ is a prime divisor of $(q-1)/(d,q-1)$ and there is a unique $R$ for each such prime $r$.   
If $\PSL(d-1,q)$ is not a nonabelian simple group then $d=3$ and $q=2$ or $q=3$. If $q=2$ then $M_\sigma\cong S_4$ and the unique possibility is $r=2$ (with unique $R$), while if $q=3$ then $M_\sigma=[3^2] \rtimes 2S_4$ (in notation of $\mathbb{ATLAS}$ \cite{Atlas}) and again the only possibility is $r=2$ (with unique $R$). Thus in all cases either $r$ is as in Line 2 of Table~$\ref{t:allr}$, or $d=3, q=2$, and $r$ is as in Line 3 of Table~$\ref{t:allr}$. 
We note that, for the cases $d=3$ with $q=2$ or $q=3$, all the conditions of Definition~\ref{def1} hold and Line 3 or 2 holds, respectively. Thus we may assume that if $d=3$ then $q\geq4$.  

Let $q=q_0^a$ with $q_0$ prime, and let $\omega\in \mathbb{F}_q$ be such that $\langle \omega \rangle= \mathbb{F}_q^*$. Since $X\leq \PGaL(d,q)$, the index $|X/(X\cap\PGL(d,q))|$ divides $a$, so let  $j:=a/|X/(X\cap\PGL(d,q))|$. Then  $X_\sigma$ is generated by $X_\sigma\cap\PGL(d,q)$ and an element $\phi^{j}x$ for some $x\in\PGL(d,q)_\sigma$, where $\phi$ is the Frobenius automorphism induced by the mapping $\phi: (a_{ij})\to (a_{ij}^{q_0})$ for matrices $(a_{ij})\in\GL(d,q)$.
Suppose first that $r=2$. Then $q$ is odd so that $o_r(q_0)=r-1$, and also $(j,r-1)=1$. Since condition $(c)$ of Definition \ref{def1} is trivial, it follows that all conditions for Line 2 of Table~$\ref{t:allr}$ hold, and  also all conditions of Definition~\ref{def1} hold. Thus we may assume further that $r$ is odd.

If $d=2$ then $M_\sigma=[q] \rtimes C_{(q-1)/(2,q-1)}$ and its unique subgroup $R$ of index $r$ is $R=[q] \rtimes C_{(q-1)/r(2,q-1)}$. We now derive the form of $R$ when $d\geq3$. 
Take $\sigma=\langle e_1\rangle$ and let $Z_0$ be the subgroup of scalar matrices in $\widehat{M}=\SL(d,q)$, as above, so that $M=\PSL(d,q)=\widehat{M}/Z_0$ and $Z_0\cong C_{(d,q-1)}$. Then
$$
\widehat{M}_\sigma=\left\{
\begin{pmatrix}
\omega^i & 0 \\
v& g
\end{pmatrix}\mid g \in \GL(d-1,q),v\in \mathbb{F}_q^{d-1},\ {\rm det}(g)=\omega^{-i}, 0\leq i< q-1
\right\}
$$
and $\widehat{M}_\sigma$ has derived subgroup
$$
(\widehat{M}_\sigma)'=\left\{
\begin{pmatrix}
1 & 0 \\
v& g
\end{pmatrix}\mid g \in \SL(d-1,q),v\in \mathbb{F}_q^{d-1}
\right\}
$$
such that $\widehat{M}_\sigma/(\widehat{M}_\sigma)'\cong C_{q-1}$. (Recall that here $d\geq3$ and if $d=3$ then $q\geq4$, so $\PSL(d-1,q)$ is a nonabelian simple group.) It follows that $\widehat{M}_\sigma$  has a unique normal subgroup of index $r$, namely the subgroup $\widehat{R}$ given by
\begin{equation}\label{e:Rhat}
\widehat{R}=\left\{
\begin{pmatrix}
\omega^{ri} & 0 \\
v& g
\end{pmatrix}\mid g \in \GL(d-1,q),v\in \mathbb{F}_q^{d-1},\ {\rm det}(g)=\omega^{-ri}, 0\leq i< (q-1)/r
\right\}.
\end{equation}
Since $Z_0\cong C_{(d,q-1)}$, we have
$$
Z_0=\left\{
\begin{pmatrix}
\omega^{i} & 0 \\
0& \omega^i I_{d-1}
\end{pmatrix}\mid  (q-1)/(d,q-1)\ \mbox{divides $i$}
\right\},
$$
and, since $r$ divides $(q-1)/(d,q-1)$, it follows that $Z_0\leq \widehat{R}$. We conclude that $R:=\widehat{R}/Z_0$ is the unique normal subgroup of $\widehat{M}_\sigma/Z_0$ of index $r$, and we have $\widehat{M}_\sigma/\widehat{R}\cong  (\widehat{M}_\sigma/Z_0)/(\widehat{R}/Z_0) \cong M_\sigma/R$. At this point we note that, if $d=2$, then $\widehat{M}_\sigma$ and $\widehat{R}$ have the same form as above and  $\widehat{M}_\sigma/\widehat{R}\cong  M_\sigma/R$. 

Recall that $X_\sigma$ is generated by $X_\sigma\cap\PGL(d,q)$ and $\phi^{j}x$ for some $x\in\PGL(d,q)_\sigma$, where $j\mid a$. Now the induced conjugation action of $X_\sigma$ on $M_\sigma/R$ is transitive on the set of  $r-1$ nontrivial cosets, and this action is equivalent to the induced action on $\widehat{M}_\sigma/\widehat{R}$ which, in turn, is equivalent to the action of $\langle \phi^j\rangle$ on $\langle \omega\rangle/\langle \omega^r\rangle$. Since $\phi^j: \omega \langle \omega^r\rangle \to \omega^{q_0^j} \langle \omega^r\rangle$, the transitivity property holds if and only if the least positive integer $\ell$ such that $\omega^{q_0^{j\ell}-1}\in  \langle \omega^r\rangle$ is $\ell=r-1$. This holds if and only if the order $o_r(q_0^j)$ of $q_0^j$ modulo $r$ is $r-1$. Since $o_r(q_0^j)$ divides $o_r(q_0)$, which divides $r-1$ this implies that $o_r(q_0)=r-1$. Moreover the order $o_r(q_0^j)=r-1$ if and only if $(j,r-1)=1$, and the transitivity condition implies that $r-1$ divides $|X_\sigma:(X_\sigma\cap \PGL(d,q))| = |X:(X\cap \PGL(d,q))|= a/j$. Thus all conditions of Line 2 of Table~\ref{t:allr} hold for this unique subgroup $R$, and if these conditions all hold then we have shown that the conditions of Definition~\ref{def1} hold. Thus the theorem is proved in the linear case.

\smallskip
\noindent {\em Case:\ unitary:}\quad Here  $M=\PSU(3,q)$ of degree  $n=q^3+1$, with $q\geq 3$, acting on totally isotropic $1$-spaces of the natural module $V= (\mathbb{F}_{q^2})^3$. The stabiliser $M_\sigma=Q\rtimes L$ where $Q$ is regular on $\Sigma\setminus\{\sigma\}$ of order $q^3$ and $L$ is cyclic of order $(q^2-1)/(3,q+1)$. The only elementary abelian quotients of $M_\sigma$ are quotients of $L$. Thus $r$ is a prime dividing $(q^2-1)/(3,q+1)$ (and $R$ is unique for such $r$).
 If $r=2$ then 
Line 4 of Table~$\ref{t:allr}$ holds since condition $(c)$ of Definition \ref{def1} is vacuous and all listed conditions (except the divisibility condition on $r$) are vacuous as $2\mid (q^2-1)/(3,q+1)$ implies that $q$ is odd.

Assume now that $r$ is odd. 
Let $q=q_0^a$ with $q_0$ prime, and $\omega$ a field element such that $\langle \omega \rangle= \mathbb{F}_{q^2}^*$. We write $\overline{x}=x^q$ for $x\in \mathbb{F}_{q^2}^*$.  Let $(\cdot, \cdot)$ be the non-degenerate unitary form on $V$ preserved by $\SU(3,q)$. Let  $\{e, x, f\}$ be a basis of $V$ such that
\begin{equation}\label{e:defUnform}
(e,e)=(f,f)=(x,e)=(x,f)=0, \text{ } (x,x)=(e,f)=1.
\end{equation}
We consider the (possibly unfaithful) action of $\widehat{M}=\SU(3,q)$ on $\Sigma$ so that, writing matrices with respect to the ordered basis $e, x, f$,
\begin{equation} \label{e:defSU}
\widehat{M}=\SU(3,q)=\{ A \in \SL(3,q^2) \mid AP\overline{A}^T=P\} \text{ where } P=\begin{pmatrix}
0 & 0&1 \\0&1&0\\1&0&0
\end{pmatrix}.
\end{equation}
Now $M=\PSU(3,q)=\widehat{M}/Z_0$, where $Z_0$ is the subgroup of scalar matrices in $\SU(3,q)$, and $|Z_0|=(3,q+1)$. Let $\sigma=\langle e\rangle$ (which is a totally isotropic $1$-space).  Then $\widehat{M}_\sigma=Q\rtimes L$ where 
$$
Q=\left\{ 
\begin{pmatrix}
1 & 0&0 \\ u&1&0\\s&-\bar{u}&1
\end{pmatrix} \mid u,s\in \mathbb{F}_{q^2}, s+\overline{s}+u\overline{u}=0 
\right\}
$$
and 
$$
L=\langle h\rangle \text{ where }  h = 
\begin{pmatrix}
\omega & 0&0 \\ 0&\omega^{q-1}&0\\0&0&\omega^{-q}
\end{pmatrix}. 
$$
We note that $Z_0<L$, $|L|=q^2-1$, and the only elementary abelian quotients of $\widehat{M}_\sigma$ are (isomorphic to) quotients of $L$. Hence the only elementary abelian quotients of $M_\sigma = \widehat{M}_\sigma/Z_0$ are (isomorphic to)  quotients of $L/Z_0$. Therefore $r$ is a prime dividing $(q^2-1)/(3,q+1)$, and the unique subgroup $R=\widehat{R}/Z_0$, where 
$\widehat{R}=Q\rtimes \langle h^r\rangle$.

Since $X\leq \PGaU(3,q)$, the index $|X/(X\cap\PGU(3,q))|$ divides $2a$, so let  $j:=2a/|X/(X\cap\PGU(3,q))|$. Then  $X_\sigma$ is generated by $X_\sigma\cap\PGU(3,q)$ and $\phi^{j}g$ for some $g\in\PGU(3,q)_\sigma$, where $\phi$ is the Frobenius automorphism induced by the mapping $\phi: (a_{ij})\to (a_{ij}^{q_0})$ for matrices $(a_{ij})\in\GU(3,q)$.  Note that $M_\sigma/R\cong \widehat{M}_\sigma/\widehat{R}=\{\widehat{R}h^i\mid 0\leq i< r-1\}$. The action of $X_\sigma$ on $M_\sigma/R$ induced by conjugation is equivalent to its action on $\widehat{M}_\sigma/\widehat{R}$. Moreover, $X_\sigma\cap\PGU(3,q)$ acts trivially, and $\phi^{j}g$ maps $\widehat{R}h^i\to \widehat{R}h^{iq_0^j}$.
This $X_\sigma$-action is transitive on the nontrivial $\widehat{R}$-cosets if and only if the least positive integer $\ell$ such that $h^{q_0^{j\ell}-1}\in  \widehat{R}$ is $\ell=r-1$. Since $h^{q_0^{j\ell}-1}\in  \widehat{R}$ if and only if $q_0^{j\ell}-1$ is divisible by $r$, the transitivity condition is equivalent to the order $o_r(q_0^j)$ of $q_0^j$ modulo $r$ being equal to $r-1$. Since $o_r(q_0^j)$ divides $o_r(q_0)$, which divides $r-1$ this implies that $o_r(q_0)=r-1$. Moreover the order $o_r(q_0^j)=r-1$ if and only if $(j,r-1)=1$, and the transitivity condition implies that $r-1$ divides $|X_\sigma/(X_\sigma\cap\PGU(3,q))|=|X/(X\cap\PGU(3,q))|
=2a/j$.  Thus all conditions of Line 4 of Table~\ref{t:allr} hold for this unique subgroup $R$, and if these conditions all hold then we have shown that the conditions of Definition~\ref{def1} hold. Thus the theorem is proved in the unitary case.

\smallskip
\noindent {\em Case:\ symplectic:}\quad Here  $M=\Sp(2d,2)$ of degree  $n=2^{2d-1} +\varepsilon 2^{d-1}$, with $\varepsilon = \pm$ and $d\geq 3$. The stabiliser $M_\sigma= O^\varepsilon(2d,2)$, and the only nontrivial abelian quotient of $M_\sigma$ is $C_2$ (with unique $R$), so  Line 5 of Table~$\ref{t:allr}$ holds. (Conditions $(a)$ and $(b)$ of Definition~\ref{def1} both hold while condition $(c)$  is vacuous.)

\smallskip
\noindent {\em Case:\ Ree or Suzuki:}\quad Here  $M={\rm Ree}(q)'$ of degree  $n=q^3+1$, with $q=3^{2a+1}\geq3$ or  $M={\rm Sz}(q)$ of degree  $n=q^2+1$, with $q=2^{2a+1}>2$. We first deal with the case  
$X={\rm Ree}(3)\cong {\rm P}\Gamma{\rm L}(2,8)$ of degree $|\Sigma|=28$, and $M={\rm Ree}(3)'\cong \PSL(2,8)$. The stabiliser $X_\sigma = C_9\rtimes C_6$ and $M_\sigma =  C_9\rtimes C_2 = D_{18}$ (see \cite[p.\,6]{Atlas}). Thus the only proper, $X_\sigma$-invariant, normal subgroup $R$ of $M_\sigma$, such that condition $(c)$ of Definition \ref{def1} holds is the subgroup  $R=C_9$ of index $r=p^c=2$. Therefore, up to equivalence, the only special pair arising from this exceptional case is $(X,R)$ with $R\cong C_9$ of index $2$ in $M_\sigma$, and  Line 7 of Table~$\ref{t:allr}$ holds.  (Conditions $(a)$ and $(b)$ of Definition~\ref{def1} both hold while condition $(c)$  is vacuous.)

Now assume  that  $q>3$ for the Ree groups so that $M={\rm Ree}(q)'= {\rm Ree}(q)$ or $M={\rm Sz}(q)$.
Then the outer automorphism group of $M$ has odd order $2a+1$ in both cases. Thus the group $X_\sigma$ cannot act transitively on the non-trivial cosets of $M_\sigma/R\cong C_r$ if $r$ is odd, and therefore $r$ is even.
The stabiliser $M_\sigma=[q^{1+1+1}] \rtimes C_{q-1}$ or $[q^{1+1}] \rtimes C_{q-1}$ respectively, so the only nontrivial abelian quotients of $M_\sigma$ are quotients of $C_{q-1}$.
Thus $r$ is a prime dividing $q-1$. Since $r$ is even, we must have $r=2$, $M={\rm Ree}(q)$, and $R$ is the unique index $2$ subgroup of $M_\sigma$. Hence Line 6 of Table~$\ref{t:allr}$ holds (Conditions $(a)$ and $(b)$ of Definition~\ref{def1} both hold while condition $(c)$  is vacuous).

\smallskip
\noindent {\em Case:\ sporadic:}\quad First let $M=M_n$ with $n\in\{11,12,22,23,24\}$. Then $M_\sigma=M_{n-1}$, and, since $M_\sigma$ must have a nontrivial abelian quotient, it follows that $n=11$ and the quotient is $C_2$ (with $R$ the unique index $2$ subgroup of $M_{10}=A_6.2$),  so  Line 8 of Table~$\ref{t:allr}$ holds. The remaining sporadic $2$-transitive groups  $M$ are $\PSL(2,11)$, $\rm{M}_{11}, A_7, \rm{HS}$, and $\rm{Co}_3$, of degrees $11, 12, 15, 176, 276$,  and with point stabilisers $M_\sigma=A_5$, $\PSL(2,11), \PSL(3,2), \PSU(3,5):2, \rm{McL}:2$, respectively. 
Since $M_\sigma$ must have a nontrivial abelian quotient it follows that $M=\rm{HS}$ or $\rm{Co}_3$ and the quotient is $C_2$ (with unique $R$), and  Line 9 or 10 of Table~$\ref{t:allr}$ holds (and we note that all conditions of Definition~\ref{def1} hold for all these cases).  

This completes the proof of Theorem~\ref{t:special}.

\bigskip
In Remark~\ref{re:Table1}, we give further details about the comments made after the statement of Theorem~\ref{t:allr}, in particular concerning  Lines $2$ and $4$ of Table~\ref{t:allr}.

\begin{remark}\label{re:Table1}
{\rm The comments below relate to Lines in Table~\ref{t:allr}, especially Lines~$2$ and $4$, which involve divisibility conditions on $r$ and $r-1$.  
\begin{enumerate}[(i)]
     \item  If $r=2$, then all conditions in Table \ref{t:allr} are vacuously satisfied, except for the divisibility condition on $r$ in Lines 2 and 4,  which forces $q=q_0^a$ to be odd in these Lines. 
    \item The divisibility conditions on $r$ imply that $q>2$ in  Lines $2$ and $4$ (and also that $q>3$ when $d$ is even in Line $2$).  
     \item In both  Lines $2$ and $4$, the divisibility condition on $r-1$ follows from the other conditions. We included this condition in Table \ref{t:allr} for clarity and for the implications given in (iv,v) below. We state the argument for Line $2$ (Line $4$ is similar). The divisibility condition on $r$ implies that $r\mid (q_0^a-1)$, that is, $q_0^a\equiv 1\pmod r$, and  so $o_r(q_0)$ divides $a=\frac{a}{j}\cdot j$. Since $o_r(q_0)=r-1$ and $(j,r-1)=1$ this implies that $(r-1)\mid (a/j)$.

\item  In Line $2$, if $r$ is odd then the  condition 
$(r-1)\mid (a/j)$ implies that $j<a$, and hence that $q=q_0^a$ is not prime.

\item In Line $4$, if $r$ is odd then  the  condition  $(j,r-1)=1$ implies that $j$ is odd (since $r-1$ is even) and hence $j\mid a$. Then the condition $(r-1)\mid 2(a/j)$ implies that  $q=q_0^a$ is not prime unless $r=3$ and $a=j=1$. In the latter case, the conditions $3\mid \frac{q^2 - 1}{(3,q+1)}$ and  $o_3(q_0)=2$ imply that  $q\equiv -1\pmod{9}$.
      
\end{enumerate}
}
\end{remark}

\subsection{Further properties of special pairs}

\begin{corollary}

\label{cor:Rtr}
Let $(X^\Sigma,R)$ be a special pair with $X, \Sigma, \sigma, R$ as in Definition~$\ref{def1}$. Then $R$ is transitive on $\Sigma\setminus\{\sigma\}$ if and only if  $X\neq{\rm Ree}(3)$. 
\end{corollary}

\begin{proof}
If  $X={\rm Ree}(3)$, then as we showed in the proof of Theorem~\ref{t:special}, $R=C_9$ and $R$ has three orbits of size $9$ on $\Sigma\setminus\{\sigma\}$. Suppose now that $X\neq{\rm Ree}(3)$.  We prove that  $R$ is transitive on $\Sigma\setminus\{\sigma\}$ with a case-by-case verification. Line numbers refer to Table~\ref{t:allr}. In Line 1, $R=C_2^2$ is regular on $\Sigma\setminus\{\sigma\}$; in Line 2, 
$R$ contains $[q^{d-1}].\SL(d-1,q)$ which is transitive on $\Sigma\setminus\{\sigma\}$; in Line 3, the group $R=A_4$ is transitive of degree $6$; in Line 4, $R$ contains a normal subgroup of order $q^3$ which is   regular on $\Sigma\setminus\{\sigma\}$; 
 in Line 5, $R=\Omega^\varepsilon(2d,2)$ is transitive on $\Sigma\setminus\{\sigma\}$ (the action is equivalent to the action on singular $1$-spaces); in Line 6,  $R$ contains a normal subgroup of order $q^3$ which is   regular on $\Sigma\setminus\{\sigma\}$; in Lines 8, 9 and 10, $R=\PSL(2,9)$, $R=\PSU(3,5)$ and $R=\rm{McL}$, respectively, each of which is transitive on the corresponding set $\Sigma\setminus\{\sigma\}$.
\end{proof}

\begin{remark}\label{rem2} {\rm
$(a)$ Lemma~\ref{lem3} provides a strategy to examine each of the special pairs $(X^\Sigma,R)$, with $X\leq \Sym(\Sigma)$ having socle $M$ and $R\lhd M_\sigma$ for some $\sigma\in\Sigma$, identified in Theorem~\ref{t:special}. Namely we first construct the innately transitive group $H$ acting on $\Omega:=[M:R]$. If we need to compute with this action it is described in a concrete way in \cite[p.46, Lemma 3.1]{PS}. 

\smallskip
$(b)$ By Theorem~\ref{t:special}, for each $2$-transitive group $X^\Sigma$, and each prime $r$, such that $X^\Sigma$ and $r$ satisfy one of the lines of Table~\ref{t:allr}, there is a unique subgroup $R$, up to conjugacy in $X$, such that  $(X^\Sigma,R)$ is a special pair. In other words, for $M=\Soc(X)$ and each $\sigma\in\Sigma$, the stabiliser $M_\sigma$ contains a unique subgroup $R$  such that $M_\sigma/R\cong C_r$. By Lemma~\ref{lem3}(b) we need to understand the normaliser $N=N_{{\rm Sym}(\Omega)}(M)$. In particular we need to find all subgroups $G$ (if any exist) such that, for $\Sigma'$ the set of $C$-orbits in $\Omega$,  
\begin{itemize}
    \item $H\leq G\leq N$, and also 
    \item the group $G^{\Sigma'}\cong G/C$  is permutationally isomorphic to $X^\Sigma$.
\end{itemize} 
}
\end{remark}

\bigskip 

 In the next two sections, we consider the lines of Table \ref{t:allr} representing infinite families of special pairs. For each line we obtain information to achieve {\bf Objectives 2} and {\bf 3}: we classify groups in $\PIT$ that are preimages of the given special pair up to permutational isomorphism and work out which of them have rank 3. The following lemma is useful for this purpose.

\begin{lemma}\label{lem4}
Using Notation $\ref{not1}$, so in particular $G^\Omega\in\PIT$, assume also that $\widehat \varphi(G^{\Omega})=(G^\Sigma, (M_\alpha)^\Sigma)$ is a special pair and $G^{\Sigma} \neq {\rm Ree}(3)$. Then the following are equivalent:
\begin{enumerate}
    \item[$(a)$] $G$ has rank $3$ on $\Omega$;
    \item[$(b)$] $G_\alpha$ is transitive on $\Omega\setminus \sigma$;
    \item[$(c)$] for some $\sigma'\in\Sigma\setminus\{\sigma\}$,  $G_{\alpha, \sigma'}$ is transitive on $\sigma'$; 
    \item[$(d)$] for some $\sigma'\in\Sigma\setminus\{\sigma\}$, $G_{\sigma, \sigma'}$ is transitive on $\sigma\times \sigma'$ (in product action);
    \item[$(e)$] for some $\sigma'\in\Sigma\setminus\{\sigma\}$, and $\alpha'\in \sigma'$, $G_{\sigma,\sigma'}=(G_{\sigma, \alpha'})(G_{\alpha,\sigma'})$.
\end{enumerate}
\end{lemma}

\begin{proof}
Note that $G$ has rank $3$ if and only if the $G_\alpha$-orbits are $\{\alpha\}, \sigma \setminus\{\alpha\}$ and $\Omega\setminus \sigma$. By Lemma~\ref{lem2b}, $G_\alpha$ is transitive on $\sigma\setminus\{\alpha\}$, and so condition $(a)$ is equivalent to condition $(b)$. 
Further, by Corollary~\ref{cor:Rtr}, $G_\sigma$ is transitive on $\Sigma\setminus\{\sigma\}$, and, since $C$ is transitive on $\sigma$, $G_\sigma=CG_\alpha$. This implies that $G_\alpha$ is transitive on 
$\Sigma\setminus\{\sigma\}$. Thus condition $(b)$ is equivalent, for (some and hence for all) $\sigma'\in\Sigma\setminus\{\sigma\}$, to $G_{\alpha, \sigma'}$ being transitive on $\sigma'$, that is, conditions $(b)$ and $(c)$ are equivalent.  Now $G_{\sigma,\sigma'}$ contains $C$ and so is transitive on both $\sigma$ and $\sigma'$. By the definition of the product action of $G_{\sigma,\sigma'}$ on $\sigma\times \sigma'$, then conditions $(c)$ and $(d)$ are equivalent. Also condition $(c)$ is equivalent to the factorisation condition in part $(e)$. 
\end{proof}

\section{Properly innately transitive groups with linear or unitary plinth}\label{s:families}

In this  section we consider properly innately transitive groups which give rise to a special pair in Lines 2 and 4 of Table \ref{t:allr}.   We start by giving constructions of group actions, which are equivalent to Construction \ref{con:gen} but are more geometric and easier to work with. Then we establish properties of each construction, and, finally,  show that all groups in $\PIT$ with  a special pair in Line $2$ or $4$ of Table~\ref{t:allr} are permutationally isomorphic to a subgroup of the group described in Construction \ref{con:psl} or \ref{con:psu}.

\subsection{Line 2: linear plinth $\PSL(d,q)$}

Let  $F=\mathbb{F}_q$, where $q=q_0^a$ with $q_0$ prime and $a\geq 1$ such that $q\geq3$, and consider $V=\mathbb{F}_q^d$, the space of $d$-dimensional row vectors.  Assume that $d\geq2$ with $(d,q)\ne (2,3)$ so that $\GL(V)$ is insoluble, and let  $\binom{V}{1}$ denote the set of $1$-subspaces of $V$.

Let $\omega\in \mathbb{F}_q$ be such that $\langle \omega \rangle= \mathbb{F}_q^*$, let $r$ be a prime diving $q-1$, and let $E=\{ \omega^i\langle \omega^r\rangle \mid 0\leq i\leq r-1\}$, the set of cosets of the multiplicative subgroup $\langle \omega^r\rangle$ of $F^*:=\langle \omega\rangle$. Note that the multiplicative group $F^*$ acts by multiplication on $E$ inducing a cyclic (and regular) permutation action with kernel $\langle \omega^r\rangle$.   Moreover,  $E\cong C_r$ with multiplication as the natural operation on $E$.
Consider the set 
\begin{equation}\label{e:defOm1}
\Omega:=\{ \varepsilon v\mid \varepsilon\in E, v\in V^*\}, \text{ where } V^*=V\setminus\{0\}.    
\end{equation}
 The elements of $\Omega$ have non-unique representations as follows.

\begin{lemma}\label{lem:rep}
$\varepsilon v = \varepsilon' v'$ if and only if, for some $\lambda\in F^*$, we have $v'=\lambda v$ and $\varepsilon'=\lambda^{-1}\varepsilon$. 
\end{lemma}

\begin{proof}
If $\lambda\in F^*$, then $\varepsilon':=\lambda^{-1}\varepsilon \in E$ (since multiplication by $\lambda^{-1}$ permutes the cosets in $E$), and we have an equality $\varepsilon v = \varepsilon' v'$ where $v'=\lambda v$.
Conversely suppose that $\varepsilon v = \varepsilon' v'$. Then $v',v$ lie in the same 1-subspace so $v'=\lambda v$ for some $\lambda\in F^*$. Then the equality $\varepsilon v = \varepsilon' v'$ implies that $\varepsilon'=\lambda^{-1}\varepsilon$.
\end{proof}

We define an action of $\GammaL(d,q)$ on $\Omega$ as follows. Each element of $\GammaL(d,q)$ is uniquely represented as a product $\phi^iA$ where $A\in\GL(d,q)$, $0\leq i<e$, and $\phi$ is the field automorphism $\lambda\to \lambda^{q_0}$ of $F$ which acts on vectors $v=(\lambda_1,\dots,\lambda_d)\in V$ by $\phi:v\to (\lambda_1^{q_0},\dots,\lambda_d^{q_0})$. Note that $\phi$ leaves invariant (setwise) the subgroup  $\langle \omega^r\rangle$ and acts on $E$ mapping $\varepsilon \in E$ to $\varepsilon^\phi:=\{\lambda^{q_0}\mid\lambda\in\varepsilon\}$. 
Therefore
\begin{equation}\label{e:action}
    \phi^i A:\varepsilon v\mapsto \varepsilon^{\phi^i} (v^{\phi^i}A), \ \mbox{for $A\in\GL(d,q)$, $0\leq i<e$, and
    $\varepsilon v\in \Omega$.}
\end{equation}
This gives us a family of proper innately transitive groups with plinths $\PSL(d,q)$. 
First, we make a formal description of this construction. In the following Proposition~\ref{p:action} we show that \eqref{e:action}
defines an action of $\GammaL(d,q)$ on $\Omega$, and that the extra condition on $r$ in Construction~\ref{con:psl} is both necessary and sufficient for this action to be properly innately transitive.

\begin{construction}\label{con:psl}
Let  $F=\mathbb{F}_q$, where $q=q_0^a\geq 3$ with $q_0$ prime and $a\geq 1$, let $V=\mathbb{F}_q^d$, the space of $d$-dimensional row vectors with $d\geq2$ and $(d,q)\ne (2,3)$. Let $r$ be a prime such that $r$ divides $(q-1)/(d,q-1)$. Let $Y=\langle \omega^r I\rangle$ and 
$$
\Omega = \{ \omega^i \langle \omega^r\rangle u|0\leq i<r,\ u\in V^*\}.
$$ 
Then the natural induced action of $\GammaL(d,q)$ on $\Omega$ given by \eqref{e:action} yields a permutation group $\ff=\GammaL(d,q)/Y$ such that $\ff^\Omega\in\PIT$ with plinth $M =\SL(d,q)Y/Y\cong \PSL(d,q)$. The group $\ff$ stabilises the partition 
$$
\Sigma=\{\sigma(U) \mid U \in \tbinom{V}{1}\}\ 
\text{ where } \sigma(U)=\{\omega^i\langle \omega^r\rangle  u\mid 0\leq i<r\} \text{ if } U=\langle u\rangle.
$$
\end{construction}

\begin{proposition}\label{p:action}
 Let $r$ be a prime dividing $q-1$, and let $\Omega$, $\Sigma, Y$ and $M$ be as defined in Construction $\ref{con:psl}$.  \begin{enumerate}
    \item[$(a)$] The map given in \eqref{e:action} defines a transitive permutation action of $\GammaL(d,q)$  on $\Omega$ with kernel $Y$, and $\GammaL(d,q)$  leaves invariant the partition $\Sigma$.
\item[$(b)$] The following are equivalent for the induced group $\ff:=\GammaL(d,q)/Y$ on $\Omega$:
        \begin{enumerate}
        \item[$(i)$] $\ff$ is innately transitive, 
        \item[$(ii)$] $r$ divides $(q-1)/(d,q-1)$, 
        \item[$(iii)$] $\ff^\Omega\in\PIT$ with plinth $M = \SL(d,q)Y/Y\cong\PSL(d,q)$ and  $C:=\C_{\overline{G}}(M)\cong C_r$.
        \end{enumerate}  
    \item[$(c)$] If the conditions in part $(b)$ hold, then $\C_{\Sym(\Omega)}(M)=C=Z/Y$ and $N_{\Sym(\Omega)}(M)=\GammaL(d,q)/Y= \ff$. 
\end{enumerate}
\end{proposition}

\begin{proof}
$(a)$ Let $\phi^iA, \phi^jB\in\GammaL(d,q)$ and $\varepsilon v\in \Omega$. Then $(\phi^iA)(\phi^jB)=\phi^{i+j} (A^{\phi^j}B)$, where $A^{\phi^j}=\phi^{-j} A \phi^j$ is the conjugate of $A$ by $\phi^j$, and we have 
\begin{align*}
(\varepsilon v)\phi^{i+j} (A^{\phi^j}B) &= \varepsilon^{\phi^{i+j}} (v^{\phi^{i+j}}A^{\phi^j}B) = \varepsilon^{\phi^{i+j}} (v^{\phi^{i}}A\phi^jB) \\
&= (\varepsilon^{\phi^{i}} (v^{\phi^{i}}A))(\phi^jB)
=(\varepsilon v)(\phi^{i}A)(\phi^jB).
\end{align*}
Since this holds for all elements of $\Omega$, it follows that \eqref{e:action} defines an action of $\GammaL(d,q)$ on $\Omega$. Since $\GammaL(d,q)$ permutes the set  $\binom{V}{1}$ of $1$-subspaces of $V$, it follows that $\GammaL(d,q)$ leaves invariant the partition $\Sigma$. Further, if $x\in\GammaL(d,q)$ acts trivially on $\Omega$ then in particular it acts trivially on $\binom{V}{1}$, and hence lies in the scalar subgroup $Z:=\langle \omega I\rangle$, and so $x=\omega^iI$ for some $i$. For each  
$\varepsilon v\in E$ we have, using \eqref{e:action} and Lemma~\ref{lem:rep},
\[
(\varepsilon v)\omega^iI= \varepsilon (\omega^i v) = (\omega^i\varepsilon)  v
\]
and $\omega^i\varepsilon = \varepsilon$ if and only if $i$ is a multiple of $r$. Thus the kernel of the action is $Y=\langle \omega^r I\rangle$.

$(b)$ The normal subgroup $\SL(d,q)$ is transitive on $\Omega$. Suppose that $Z\cap \SL(d,q)=Y\cap \SL(d,q)$. Then the permutation group on $\Omega$ induced by $\SL(d,q)$ is 
\[
M:=\SL(d,q)Y/Y\cong \SL(d,q)/(Y\cap \SL(d,q))=\SL(d,q)/(Z\cap \SL(d,q))\cong \PSL(d,q),
\]
which is a nonabelian simple group since  $(d,q)\ne (2,2)$ or $(2,3)$. In this case $\ff$ is innately transitive with plinth $M\cong \PSL(d,q)$. On the other hand if $Z\cap \SL(d,q)\ne Y\cap \SL(d,q)$ then $1\ne (Z\cap \SL(d,q))/(Y\cap \SL(d,q)) \cong (Z\cap \SL(d,q))Y/Y$, a normal subgroup of $\ff$ contained in $\SL(d,q)Y/Y$. Since $Z/Y\cong C_r$ and $r$ is prime, it follows that   $(Z\cap \SL(d,q))Y/Y=Z/Y\cong C_r$ and $r$ divides $|Z\cap\SL(d,q)|=(d,q-1)$, the order of the Schur multiplier of $\PSL(d,q)$. In this case $Z/Y$ is the unique minimal normal subgroup of $\ff$; it is intransitive and $\ff$ is not innately transitive. Thus $\ff$ is innately transitive if and only if $Z\cap \SL(d,q)=Y\cap \SL(d,q)$, which is equivalent to $(d,q-1)$ dividing $|Y|=(q-1)/r$, and in turn this is equivalent to $r$ dividing $(q-1)/(d,q-1)$. In this case the group induced by $Z$ on $\Omega$ is $Z/Y\cong C_r$, and this group centralises the plinth $M$; we deduce that $\ff\in\PIT$ with $C_{\ff}(M)\cong C_r$. This proves part $(b)$.

 $(c)$  Let  $\alpha=\langle \omega^r\rangle e_1\in \Omega$,  so  $\alpha\in \sigma(U)$ for  $U=\langle e_1\rangle\in\binom{V}{1}$. Then $\SL(d,q)_{\sigma(U)} = \SL(d,q)_U$,  and induces a transitive cyclic action on the nonzero elements of $U$. Thus the stabiliser $\SL(d,q)_{\alpha}$ is a nontrivial normal subgroup of prime index $r$ in $\SL(d,q)_U$.  It follows from \cite[Theorem 3.2(i)]{PS} that $\C_{\Sym(\Omega)}(M)\cong M_\sigma/ M_\alpha\cong C_r$. Thus $\C_{\Sym(\Omega)}(M)$ is isomorphic to $C$ and to $Z/Y$. Since $Z/Y$ centralises $M$, we have $Z/Y\leq \C_G(M)\leq \C_{\Sym(\Omega)}(M)$, so equality holds and the first assertion is proved.
  By definition, $N_{\Sym(\Omega)}(M) \ge \GammaL(d,q)/Y.$ On the other hand 
$$
|N_{\Sym(\Omega)}(M)/C| \le |N_{\Sym(\Sigma)}(M^{\Sigma})|= |\PGammaL(d,q)|=|(\GammaL(d,q)/Y)/C|,
$$ 
where the first equality holds by \cite[Table 7.4]{Cam99} since $M^{\Sigma}$ is $2$-transitive. Thus $N_{\Sym(\Omega)}(M) = \GammaL(d,q)/Y.$
\end{proof}

Next we determine the subgroups of $\overline{G}$ in Construction~\ref{con:psl} that yield special pairs and those that lie in $\PIT_3$. By Lemma~\ref{lem3}(b), all such subgroups must contain $C\times M$.

\begin{lemma}\label{lem:sppirSL}
Let $\ff, M$ be as in Construction $\ref{con:psl}$ and let $C$ be as in Proposition $\ref{p:action}$. Assume that $C \times M \le G \le \ff.$
\begin{enumerate}
        \item[$(a)$] Let  $\alpha=\langle \omega^r\rangle e_1\in \Omega$ and let $R=(\SL(d,q)_{\alpha})^\Sigma$. Then $(G^\Sigma,R)$ is a special pair if and only if the conditions in Line $2$ of Table $\ref{t:allr}$ hold for $G^{\Sigma}$.
    \item[(b)] If $(d,r)\ne (2,2)$, then $(G^\Sigma,R)$ is a special pair
    if and only if  $G^{\Omega} \in\PIT_3$.
    
    \item[(c)] If $(d,r)= (2,2)$, then $(G^\Sigma,R)$ is a special pair
    if and only if $q\equiv 1\pmod 4$. Moreover,  if $(d,r)= (2,2)$ and $q\equiv 1\pmod 4$, then either     
    \begin{enumerate}
            \item[$(i)$] $G^\Sigma\not\leq \PSigmaL(2,q)$ and $G^{\Omega} \in\PIT_3$; or
            \item[$(ii)$] $G^\Sigma\leq \PSigmaL(2,q)$ and $G^{\Omega}$ has rank $4$.
        \end{enumerate}
 \end{enumerate}
\end{lemma}

\begin{proof}
 $(a)$   Note that the actions of $G$ on the set $\tbinom{V}{1}$ of $1$-spaces and on $\Sigma$ are permutationally isomorphic. Also the subgroup $R$ is the group $\widehat{R}/Z_0$ in \eqref{e:Rhat} in the proof of Theorem~\ref{t:special}. Thus by Theorem \ref{t:special}, $(G^\Sigma,R)$ is a special pair if and only if the conditions in Line 2 of Table \ref{t:allr} hold for $G^{\Sigma}$. (Note that Line 3  of Table \ref{t:allr}  does not apply since $q\geq3$ in Construction~\ref{con:psl}.)

    $(b)$ and $(c)$\quad If $G^{\Omega} \in\PIT_3$, then,  by Lemma~\ref{lem2}, $(G^\Sigma,R)$ is a special pair, so by part $(a)$ the conditions in Line 2 of Table \ref{t:allr} hold for $G^{\Sigma}$. Conversely, suppose that  the conditions in Line 2 of Table \ref{t:allr} hold for $G^{\Sigma}$, so that  $(G^\Sigma,R)$ is a special pair by part $(a)$, with $R:=(\SL(d,q)_{\alpha})^\Sigma$, and $\alpha=\langle \omega^r\rangle e_1\in \Omega$ lying in the block $\sigma=\sigma(\langle e_1\rangle)$ of $\Sigma$. 
     By Lemma~\ref{lem4}, $G^\Omega$ has rank 3 if and only if, for some $\sigma'\in \Sigma\setminus\{\sigma\}$, $G_{\alpha,\sigma'}$ is transitive on $\sigma'$. We may take  $\sigma'=\sigma(\langle e_d\rangle)$, and we may consider the (possibly unfaithful) action of $\widehat{G}_{\alpha,\sigma'}$, where $\widehat{G}$ is the full preimage of $G$ in $\GammaL(d,q)$ (since $G$ and $\widehat{G}$ induce the same group on $\Omega$).
     Let $L=\widehat{G}_{\alpha, \sigma'}$, the stabiliser of $\alpha$ and $\sigma'$. 
     
     We claim that, if $(d,r)\ne (2,2)$, then $L\cap \GL(d,q)$ is transitive on $\sigma'$. As discussed above, this condition implies that $G^\Omega$ has rank $3$ so $G^\Omega\in\PIT_3$, and hence a proof of this claim will complete the proof of part (b).
    Note that, since $\widehat{G}$ contains $\SL(d,q)$, the group $L\cap \GL(d,q)$ contains all matrices $\mathcal{A}$ in $\SL(d,q)$ of the form 
    \[
    \mathcal{A}=
   \begin{pmatrix}
    y & 0 & 0 \\
   v & A & w  \\
   0 & 0&x
    \end{pmatrix},  \mbox{where $ v,w\in F^{(d-2)\times 1},x \in F^*, y \in \langle \omega^r\rangle$, and $A\in\GL(d-2, q)$,}
    \]
    with $v$, $w$, $A$ empty matrices if $d=2$. The matrix $\mathcal{A}\in\SL(d,q)$ if and only if $yx\det(A)=1$, where we take $\det(A)=1$ for an empty matrix $A$. In particular, if $d\geq3$, then each $x \in F^*$ occurs in such a matrix $\mathcal{A}$ (by choosing $y=1$ and $\det(A)=x^{-1}$), and hence $L\cap \GL(d,q)$ is transitive on $\sigma'$. Hence the claim is proved if $d\geq3$. So assume that $d=2$. 
    The preimage of $C\times M$ in $\widehat{G}$ is the subgroup $Z\,\SL(2,q)$, and the elements of $(Z\,\SL(2,q))_{\sigma,\sigma'}$ are the matrices of the form 
    $$
    \begin{pmatrix}
z & 0 \\
0 & z  \\
\end{pmatrix}\cdot \begin{pmatrix}
x^{-1} & 0 \\
0 & x  \\
\end{pmatrix}=\begin{pmatrix}
zx^{-1} & 0 \\
0 & zx  \\
\end{pmatrix} \text { for } z, x \in F^*.
$$
Such a matrix fixes $\alpha$ if and only if $y:=zx^{-1}\in\langle \omega^r\rangle$. Thus the matrices in  $(Z\,\SL(2,q))_{\alpha,\sigma'}$ are those of the form   
    
    \begin{equation}\label{e:dr2}
     \begin{pmatrix}
y & 0 \\
0 & yx^2  \\
\end{pmatrix} \text { for } y\in\langle \omega^r\rangle, x \in F^*.   
    \end{equation}
Since matrices $yI_2$ act trivially on $\sigma'$, it follows that the action of $(Z\,\SL(2,q))_{\alpha,\sigma'}$ on $\sigma'=\sigma(\langle e_2\rangle)$
is equivalent to the action of $\langle \omega^2\rangle$ on $F^*/\langle \omega^r\rangle$.
If $r$ is odd then this action is transitive. Hence $L\cap \GL(d,q)$ is transitive if $d=2$ and $r$ is odd. Thus the claim is proved, and we have proved part (b). 

Assume now that $(d,r)=(2,2)$. In this  case, the conditions in Line 2 of Table \ref{t:allr} reduce simply to $q\equiv 1\pmod{4}$ since $r=2$ divides $(q-1)/(2,q-1)$. Thus the first assertion of part (c) follows from part (a). Assume now that $q\equiv 1\pmod{4}$. Our arguments in the previous paragraph have shown that $(Z\,\SL(2,q))_{\alpha,\sigma'}$ fixes each of the two points $\beta=\langle \omega^2\rangle e_2$ and $\beta'=\langle \omega^2\rangle \omega e_2$ of $\sigma'$. Further, $\GammaL(2,q)_{\sigma,\sigma'} $ consists of all elements of the form
\begin{equation}\label{e:dr22}
 \phi^i \begin{pmatrix}
x' & 0 \\
0 & x  \\
\end{pmatrix} \text { for }  x,x'\in F^*, 0\leq i<a.
\end{equation}
 Such an element fixes $\alpha$ if and only if $x'\in \langle \omega^2 \rangle$, and interchanges 
$\beta$ and $\beta'$ if and only if $x\in F^*\setminus \langle \omega^2 \rangle$. Since $\widehat{G}$ contains all elements as in \eqref{e:dr2}, $\widehat{G}_{\alpha,\sigma'}$ contains an element interchanging $\beta$ and $\beta'$ if and only if $\widehat{G}$ contains an element of the form in \eqref{e:dr22}, for some $i$, with $x'=1$ and $x=\omega$.
This property holds if and only if $G^\Sigma\not\leq \PSigmaL(2,q)$. Thus if $G^\Sigma\not\leq \PSigmaL(2,q)$ then $G^\Omega\in\PIT_3$, and (c)(i) holds. Assume finally that $G^\Sigma\leq \PSigmaL(2,q)$. Then $\widehat{G}_{\alpha,\sigma'}$ fixes the two points $\beta$ and $\beta'$ of $\sigma'$. Since $\widehat{G}_{\alpha}$ is transitive on $\Sigma\setminus\{\sigma\}$, it follows that 
$\widehat{G}_{\alpha}$ has two orbits in $\Omega\setminus\sigma$, and thus $G^\Omega$ has rank $4$, as in (c)(ii). This completes the proof. 
\end{proof}

Next we show that all properly innately transitive groups corresponding to special pairs in Line 2 of Table~\ref{t:allr} arise from Construction~\ref{con:psl}.

\begin{lemma}\label{lem:pslaction}
Suppose that  $G_0 \le \Sym(\Omega_0)$ is such that $G_0^{\Omega_0} \in \PIT$ with plinth $M_0 \cong \PSL(d,q)$, and that $\widehat{\varphi}(G_0^{\Omega_0})=(G_0^{\Sigma_0},R_0^{\Sigma_0})$ (for $\Sigma_0$ the orbit set of $C_0=\C_{\Omega_0}(M_0)$) is a special pair satisfying the conditions of Line $2$ of Table $\ref{t:allr}$. 
Then $G_0^{\Omega_0}$ is permutationally isomorphic to a subgroup $G^\Omega$ where $C \times M \le G\le \ff$, and 
$\ff, M, r,$ and $\Omega$ are as in Construction $\ref{con:psl}$.
\end{lemma}

\begin{proof}
 Let $V:=(\mathbb{F}_q)^d=\langle e_1, \ldots, e_d \rangle$, and note that $G_0^{\Sigma_0}=G_0/C_0\leq \PGaL(d,q)$ 
 and $G_0^{\Sigma_0}$ is permutationally isomorphic to its action on the set $\binom{V}{1}$ of $1$-spaces. Thus  we may identify $\Sigma_1$ with $\binom{V}{1}$, and without loss of generality we assume that the $1$-space $\sigma_0 \in \Sigma_0$ fixed by $R_0^{\Sigma_0}$ is $\sigma_0 =\langle e_1 \rangle.$  By the conditions of Line $2$ of Table $\ref{t:allr}$,  $r :=|(M_0)_{\sigma_0}/R_0|$ is  a prime dividing $(q-1)/(d,q-1)$ and $o_r(q_0)=r-1$, where $q=q_0^a$ with $q_0$ prime, and $(d,q)\ne (2,2), (2,3)$. 
 
 Thus the assumptions of  Proposition \ref{p:action} hold and it follows that  $M_0$ acts faithfully and transitively on the set $\Omega$ defined in \eqref{e:defOm1} via the action in \eqref{e:action} with $R_0$ the stabiliser of a point of $\Omega$. Therefore we may assume that $M_0^{\Omega}=M$, as in Construction~\ref{con:psl}.  Note that, by \eqref{e:phi2}, $R_0$ is the stabiliser in $G_0$ of a point in $\Omega_0$ lying in $\sigma_0$, so the actions of $M_0$ on $\Omega_0$ and $\Omega$ are permutationally isomorphic, so we may assume further that $\Omega_0=\Omega$ and $M_0=M$.
 Under this identification $G_0$ (which normalises $M_0$) is a subgroup of $N_{\Sym(\Omega)}(M)$ which, by Proposition~\ref{p:action}$(c)$, is the group $\overline{G}$ of Construction~\ref{con:psl}. Finally, $\C_{G_0}(M_0)$ is nontrivial by \eqref{e:cent}, and, since $C=\C_{\Sym(\Omega)}(M)\cong C_{r}$ has prime order, by Proposition \ref{p:action}$(c),$ it follows that $\C_{G_0}(M_0)= C$, and we have $C \times M\leq G_0\leq\overline{G}$, completing the proof.
\end{proof}

Finally we make a comment on the structure of the group $\overline{G}$ in Construction~\ref{con:psl}, which had puzzled us for a while.

\begin{remark}\label{rem:psl}
{\rm 
Suppose that the conditions in  Proposition~\ref{p:action}~$(b)$ hold 
so that $\ff=\GammaL(d,q)/Y$ is innately transitive. The structure of $\ff$ is not immediately obvious, though by Proposition~\ref{p:action}~$(b)(iii)$, $\ff$ has a normal subgroup $C_r\times \PSL(d,q)$ which itself lies in $\PIT$. It turns out that $\ff=(C_r\times\PGL(d,q)).\langle \phi\rangle$ if and only if $r$ does not divide $d$. 

To see this, we note that in all cases 
\[
\ff=\GammaL(d,q)/Y=(\GL(d,q).\langle \phi\rangle)/Y=(\GL(d,q)/Y).\langle \phi\rangle,
\]
and the map $f:  A \to (\det(A)\langle \omega^r\rangle,ZA)$ defines a group homomorphism $f:\GL(d,q)\to E\times \PGL(d,q)$. Thus $\GL(d,q)/{\rm Ker}(f)$ is isomorphic to a subgroup of $ E\times \PGL(d,q)\cong  C_r\times \PGL(d,q)$. Now  ${\rm Ker}(f)$  consists of the scalar matrices $\lambda I$ such that $\det(\lambda I)=\lambda^d\in \langle \omega^r\rangle.$
Hence, if $r$ does not divide $d$, then ${\rm Ker}(f)=Y$, and, since  $|\GL(d,q)/Y|=|Z/Y|\cdot |\GL(d,q)|/Z|=|C_r\times \PGL(d,q)|$, it follows that $f$ is onto and we have $\GL(d,q)/Y\cong C_r\times \PGL(d,q)$. On the other hand if $r$ divides $d$ then ${\rm Ker}(f)=Z$ and $(\GL(d,q))f= 1\times \PGL(d,q)$, and $\GL(d,q)/Y$ does not have the structure of a direct product. 
}
\end{remark}

\subsection{Line 4: plinth $\PSU(3,q)$}
Let $q=q_0^a$ where $q_0$ is a prime. Let $\langle \omega \rangle= \mathbb{F}_{q^2}^*$ and let $\overline{\alpha}=\alpha^q$ for $\alpha$ in $\mathbb{F}_{q^2}$.
Consider $V={(\mathbb{F}_{q^2})}^3$, the space of $3$-dimensional row vectors, as the natural module for $\SU(3,q)$, and let $(\cdot, \cdot)$ be the corresponding non-degenerate unitary form on $V$. Then $V$ has a basis  $\{e, x, f\}$ such that
$$ 
(e,e)=(f,f)=(x,e)=(x,f)=0, \text{ } (x,x)=(e,f)=1,
$$
and we write the elements of  $\SU(3,q)$ relative to this basis.
Then
$$
\SU(3,q)=\{ A \in \SL(3,q^2) \mid AP\bar{A}^T=P\} \text{ where } P=\begin{pmatrix}
0 & 0&1 \\0&1&0\\1&0&0
\end{pmatrix}.
$$ 
Let $\mathcal{T}$ be the set of all $v \in V^*$ such that $(v,v)=0.$ Let $\binom{\mathcal{T}}{1}$ denote the set of totally isotropic 1-subspaces of $V$, so $\langle v \rangle \in \binom{\mathcal{T}}{1}$ if and only if $v \in \mathcal{T}$.

Let $Z$ be the subgroup of all scalar matrices of $\GL(3,q^2).$
We define $\GammaU(3,q)$ to be the group of semisimilarities of the unitary space $V$, that is to say, $g \in \GammaU(3,q)$ if and only if $g \in \GammaL(3,q^2)$ and there exists $\lambda \in \mathbb{F}_{q^2}^*$ and $\alpha \in \Aut (\mathbb{F}_{q^2})$ such that 
$$
(v g, u g)= \lambda (v,u)^{\alpha} \text{ for all } v,u \in V.
$$
By \cite[\S 2.3]{KL}, $\GammaU(3,q) = (\GU(3,q)Z) \rtimes \langle \phi \rangle $ where   $\phi$ is the field automorphism $\lambda\to \lambda^{q_0}$ of $F$ which acts on vectors $v=(\lambda_1,\lambda_2,\lambda_3)\in V$ by $\phi:v\to (\lambda_1^{q_0},\lambda_2^{q_0},\lambda_3^{q_0})$.

Let $r$ be a prime diving $q^2-1$, and let $E=\{ \omega^i\langle \omega^r\rangle \mid 0\leq i\leq r-1\}$. 
Consider the set 
\begin{equation}\label{e:defOm2}
\Omega:=\{ \varepsilon v\mid \varepsilon\in E, v\in \mathcal{T}\}.
\end{equation}
 The elements of $\Omega$ have non-unique representations as in Lemma \ref{lem:rep} with $F= \mathbb{F}_{q^2}.$
We define an action of $\GammaU(3,q)$ on $\Omega$ as follows. Each element of $\GammaU(3,q)$ is uniquely represented as a product $\phi^i A$ where $A \in \GU(3,q) Z$ and $0 \le i < 2a$, and we define
\begin{equation}\label{e:actionU}
    \phi^i A:\varepsilon v\mapsto \varepsilon^{\phi^i} (v^{\phi^i}A), \ \mbox{for $A\in\GU(3,q)Z$, $0\leq i<2a$, and
    $\varepsilon v\in \Omega$.}
\end{equation}
We prove that this gives a family of proper innately transitive groups with plinth $\PSU(3,q).$ First, we make a formal description of this construction. Then we state Proposition \ref{p:actionU} recording the properties of the action of $\GammaU(3,q)$ on $\Omega$ via \eqref{e:actionU} similar to the ones we proved for the linear case in Proposition \ref{p:action}. We omit the proof of Proposition \ref{p:actionU} since it is fully analogous to the proof of Proposition \ref{p:action}.

\begin{construction}\label{con:psu}
Let  $q \ge 3$ and let $r$ be a prime dividing $(q^2-1)/(3,q+1)$. Let $Y=\langle \omega^r I\rangle$, and let $\Omega$ be the set defined in \eqref{e:defOm2}. Then the natural induced action of $\GammaU(3,q)$ on $\Omega$ given by \eqref{e:actionU} yields a permutation group $\ff=\GammaU(3,q)/Y$ such that $\ff^\Omega \in\PIT$ with plinth $M=\SU(3,q)Y/Y\cong \PSU(3,q)$. The group $\ff$ stabilises the partition of $\Omega$
\[
\Sigma=\{\sigma(U) \mid U\in  \tbinom{\mathcal{T}}{1}\},\ \text{ where }   \sigma(U)=\{\omega^i\langle \omega^r\rangle u\mid 0\leq i<r\} \text{ if } U=\langle u \rangle.
\]
\end{construction}

\begin{proposition}\label{p:actionU}
Let $r$ be a prime dividing $q^2-1$, and let $q, \Omega, \Sigma, Y$ and $M$ be as in Construction~$\ref{con:psu}$.
\begin{enumerate}
    \item[$(a)$] Then the map given in \eqref{e:actionU} defines a transitive permutation action of $\GammaU(3,q)$  on $\Omega$ with kernel $Y$, and $\GammaU(3,q)$  leaves invariant the partition $\Sigma$ of $\Omega$.
    \item[$(b)$] The following are equivalent for the induced group $\ff:=\GammaU(3,q)/Y$ on $\Omega$:
        \begin{enumerate}
        \item[$(i)$] $\ff^\Omega$ is innately transitive, 
        \item[$(ii)$] $r$ divides $(q^2-1)/(3,q+1)$, 
        \item[$(iii)$] $\ff^\Omega\in\PIT$ with plinth $M = \SU(3,q)Y/Y\cong\PSU(3,q)$ and  $C:=\C_{\ff}(M)\cong C_r$.
        \end{enumerate}
\item[$(c)$] If the conditions in part $(b)$ hold, then $\C_{\Sym(\Omega)}(M)=C=Z/Y$ and $N_{\Sym(\Omega)}(M)=\GammaU(3,q)/Y= \ff.$
  \end{enumerate}
\end{proposition}

In the following lemma we determine the subgroups of $\overline{G}$ in Construction~\ref{con:psu} that yield special pairs and those that lie in $\PIT_3$.
All such subgroups must lie in $\PIT$, and by  Lemma~\ref{lem3}(b), they must contain $C\times M$.

\begin{lemma}\label{lem:sppirSU}
Let $\overline{G}, M$ be as in Construction $\ref{con:psu}$ and let $C$ be as in Proposition $\ref{p:actionU}$. 
Assume that $C \times M \le G \le \ff.$ Let  $\alpha=\langle \omega^r\rangle e \in \Omega$ and let $R=(\SU(3,q)_{\alpha})^\Sigma$.
\begin{enumerate}
\item[$(a)$]  Then $(G^\Sigma,R)$ is a special pair if and only if the conditions in Line $4$ of Table $\ref{t:allr}$ hold for $G^{\Sigma}$.
    
    \item[$(b)$] If $(G^\Sigma,R)$ is a special pair (so in particular $r$ divides $q^2-1$), then either
    \begin{enumerate}
        \item[(i)]  $r$ is odd and $r$ divides $q-1$, and in this case $G^\Omega \in\PIT_3$; or
        \item[(ii)] $r$ divides $q+1$, and $G^\Omega$ has rank $4$.
    \end{enumerate}
\end{enumerate}
\end{lemma}

\begin{proof}
  $(a)$ The proof is fully analogous to the proof of Lemma \ref{lem:sppirSL}$(a)$, and so is omitted.

     $(b)$  Suppose now that $(G^\Sigma,R)$ is a special pair so, by part (a), the conditions in Line $4$ of Table $\ref{t:allr}$ hold for $G^{\Sigma}$, and in particular $r$ divides $(q^2-1)/(3,q+1)$. Hence, by Proposition~\ref{p:actionU}$(b)$, $\overline{G}^\Omega\in\PIT$
    with   plinth $M=\SU(3,q)Y/Y$ and  $C=\C_{\overline{G}}(M)\cong C_r$. Since $C\times M\leq G\leq \overline{G}$, it follows that $G^\Omega\in\PIT$ with the same plinth $M$ and with $\C_G(M)=C\cong C_r$.
     
     Now $\alpha=\langle \omega^r\rangle e\in \Omega$,  $\alpha$ lies in the block $\sigma=\sigma(\langle e\rangle)$ of $\Sigma$, and  $R:=(\SU(3,q)_{\alpha})^\Sigma$.
     Let  $\sigma'=\sigma(\langle f\rangle)$, so $\sigma'\in\Sigma\setminus\{\sigma\}$. By Lemma~\ref{lem4}, $G^\Omega$ has rank $3$ if and only if $G_{\alpha,\sigma'}$ is transitive on $\sigma'$. 
     We claim that this transitivity condition holds if and only if $r$ is odd and divides $q-1$, that is to say, if and only if part (b)(i) holds.

   Let $\widehat{G}$ be the full preimage of $G$ in $\GammaU(3,q)$. Since $G$ and $\widehat{G}$ induce the same group on $\Omega$, 
   $G_{\alpha,\sigma'}$ is transitive on $\sigma'$ if and only if $\widehat{G}_{\alpha,\sigma'}$ is transitive on $\sigma'$.
   
    Note that each element of $\GammaU(3,q)_{\alpha, \sigma'}$ also fixes $\langle e, f\rangle^\perp=\langle x\rangle$, and therefore $\GammaU(3,q)_{\alpha, \sigma'}$  consists of all elements $u$ of the form
    \begin{equation}
    \phi^i \cdot \lambda I \cdot
    \begin{pmatrix}\label{e:ULelt}
    z & 0 & 0 \\
    0 & y & 0 \\
    0 & 0  &  z^{-q}
    \end{pmatrix}, \mbox{where $0 \le i< 2a$, $z, y, \lambda \in \mathbb{F}_{q^2}$, $\lambda \cdot  z \in \langle \omega^r\rangle, y^{q+1}=1$.}
    \end{equation}
    Since $\lambda \cdot  z \in \langle \omega^r\rangle$, we have $\lambda \cdot  z^{-q} \in z^{-q-1}\langle \omega^r\rangle$. 
    Thus this element $u$ maps the point $\beta= \langle \omega^r\rangle f\in \sigma'$ to $\beta u=z^{-q-1}\langle \omega^r\rangle f$.
    Assume first that $r$ does not divide $q+1$, or equivalently, that $r$ is odd and  $r \mid (q-1)$. Then $(Z\,\SU(3,q))_{\alpha, \sigma'}$, and hence also $\widehat{G}_{\alpha, \sigma'}$, contains the element $u$ as above with $i=0, \lambda=\omega^{-1}$, $z=\omega$, $y=\omega^{q-1}$.  This element $u$ maps an arbitrary point $\omega^j \langle \omega^r\rangle f\in\sigma'$ to $\omega^{j-q-1}\langle \omega^r\rangle f = \omega^{j-2}\langle \omega^r\rangle f$. Hence the action of $u$ on $\sigma'$ is equivalent to the action of $\omega^{-2}$ by multiplication on $F^*/\langle \omega^r\rangle$. Since $r$ is odd it follows that  $\langle \omega^{-2}\rangle$ is transitive on $F^*/\langle \omega^r\rangle$, and hence $\widehat{G}_{\alpha, \sigma'}$ is transitive on $\sigma'$. This completes the proof of  (b)(i): if $r$ is odd and $r$ divides $q-1$ then $G^\Omega\in\PIT_3$.

    As we noted above, the condition `$r$ odd and $r$ divides $q-1$' is equivalent to `$r$ does not divide $q+1$'. So assume now that $r$ does divide $q+1$.  
    Note $\sigma'$ consists of the $r$ points $\beta_j:=\omega^j\langle \omega^r\rangle f$ of $\Omega$, for $0\leq j\leq r-1$, and  a typical element $u\in G_{\alpha,\sigma'}$ is of the form given in \eqref{e:ULelt} and so maps a point $\beta_j$ to $\beta_ju=(\omega^j)^{\phi^i}z^{-q-1} \langle \omega^r\rangle f$. Since $r$ divides $q+1$, we have 
    $z^{-q-1} \in \langle \omega^r\rangle$. Also $(\omega^j)^{\phi^i} = \omega^{jq_0^i}$, where $q$ is a power of the prime $q_0$. Thus $\beta_ju= \omega^{jq_0^i} \langle \omega^r\rangle f$. In particular $u$ fixes $\beta_0=\beta$, so $G_{\alpha,\sigma'}$ fixes $\beta_0$, and the conditions in the last column of Line $4$ of Table~\ref{t:allr}  ensure that the group $G_{\alpha,\sigma'}$ acts transitively on the remaining $r-1$ points of $\sigma'\setminus\{\beta_0\}$. To see this explicitly we refer to the last paragraph of the proof of the unitary case of Theorem~\ref{t:special}, and in particular to the element $g$ which in our notation lies in $G_\sigma\cap \PGU(3,q)$. Let $\widehat{g}\in\widehat{G}_\sigma$ such that $\widehat{g}^\Sigma = g$. Since $M\leq G$ and $M^\Sigma$ is $2$-transitive, we may assume that $g$ lies in $G_{\sigma,\sigma'}$, that is to say, $\widehat{g}\in\widehat{G}_{\sigma,\sigma'}$. 
    Further, since $C<G_{\sigma,\sigma'}$ and $C$ is transitive on $\sigma$, we may assume further that $\widehat{g}\in\widehat{G}_{\alpha,\sigma'}$. Then the same argument given for the transitivity of $\langle g\rangle$ on $\sigma\setminus\{\alpha\}$, in the proof of the unitary case of Theorem~\ref{t:special}, also shows that  $\langle \widehat{g}\rangle$ is transitive on $\sigma'\setminus\{\beta_0\}$, and hence that $G_{\alpha,\sigma'}$ is transitive on $\sigma'\setminus\{\beta_0\}$.
    Now, arguing as in the proof of Lemma~\ref{lem4}, the number of $G_\alpha$-orbits in $\Omega\setminus\sigma$ is equal to the number of $G_{\alpha,\sigma'}$-orbits in $\sigma'$, and we have just shown that this number is two. Since the $G_\alpha$-orbits in $\sigma$ are $\{\alpha\}$ and $\sigma\setminus \{\alpha\}$ it follows that $G_\alpha$ has exactly four orbits in $\Omega$ and hence that $G^\Omega$ has rank $4$. This completes the proof of  (b)(ii): if $r$ divides $q+1$ then $G^\Omega$ has rank $4$.
    \end{proof}

The proof of the following lemma is fully analogous to the proof of Lemma \ref{lem:pslaction}, and we omit it. It shows that all properly innately transitive groups corresponding to special pairs in Line 4 of Table~\ref{t:allr} arise from Construction~\ref{con:psu}.

\begin{lemma}\label{lem:psUaction}
Suppose that  $G_0 \le \Sym(\Omega_0)$ is such that $G_0^{\Omega_0} \in \PIT$ with plinth $M_0 \cong \PSU(3,q)$, and that $\widehat{\varphi}(G_0^{\Omega_0})=(G_0^{\Sigma_0},R_0^{\Sigma_0})$ (for $\Sigma_0$ the orbit set of $C_0=\C_{\Omega_0}(M_0)$) is a special pair satisfying the conditions of Line $4$ of Table $\ref{t:allr}$. 
Then $G_0^{\Omega_0}$ is permutationally isomorphic to a subgroup $G^\Omega$ where $C \times M \le G\le \ff$, and 
$\ff, M, r,$ and $\Omega$ are as in Construction $\ref{con:psu}$.
\end{lemma}

\section{Properly innately transitive groups with symplectic or ${\rm Ree}$ plinths 
} \label{sect:spree}

In this section we consider groups in $\PIT$ giving rise to the special pair
in Lines 5 and 6 of Table \ref{t:allr}.

\subsection{Line $5$: plinth $\Sp(2d,2)$}
Let $M=\Sp(2d,2)$ with $d \ge 3$ and let $V=(\mathbb{F}_2)^{2d}$ be the natural module for $M.$ Let $(\cdot, \cdot)$ be the corresponding bilinear alternating form that is also symmetric since  $V$ is over a field of even characteristic. Here $\Sigma$ is the set of quadratic forms $Q$  on $V$ of type $\varepsilon\in\{+,-\}$ such that $$(u,v)=Q(u+v)- Q(u) - Q(v) \text{ for all } u,v \in V.$$ 
For $g\in M$ the action on $\Sigma$ is defined by $Q^g(u)=Q(ug^{-1})$ for all $u \in V.$

Let $\sigma=Q$ be a quadratic form in $\Sigma.$
Then $M_\sigma= O(V,Q)$ is the orthogonal group preserving $Q$ on $V$, which is isomorphic to $O^{\varepsilon}(2d,2).$ Therefore $M_\sigma$ is almost simple with socle isomorphic to $\Omega^{\varepsilon}(2d,2)$. Let $R$ be this socle, so that $r=|M_\sigma/R|=2$. 
Note that $N_{\Sym(\Sigma)}(M^\Sigma)=M^\Sigma$, and so if $(X^\Sigma,R)$ is a special pair as in Line $5$ of Table \ref{t:allr}, then $X^\Sigma=M^\Sigma$.

In this case, we  use Construction~\ref{con:gen}. So, $\Omega = \{Rx \mid x \in M\}$ with $M$ acting  by right multiplication,  $C=\C_{\Sym(\Omega)}(M)$, and we identify $\Sigma$ with the set of $C$-orbits.
%
%
The following lemma follows directly from Lemma \ref{lem3} and Corollary~\ref{lem:uniq} since $N_{\Sym(\Sigma)}(M^\Sigma)=M^\Sigma$. 
\begin{lemma}\label{lem:spaction}
With the notation of Construction~$\ref{con:gen}$ for $M=\Sp(2d,2)$, $H^\Omega\in\PIT$ and $\widehat{\varphi}(H^\Omega)=(M^\Sigma, R)$. Moreover, if $G^{\Omega_0} \in \PIT$ with plinth isomorphic to $\Sp(2d,2)$ and $\widehat{\varphi}(G^{\Omega_0})\approx (M^\Sigma, R)$, for some quadratic form $Q$, then $G^{\Omega_0}$ is permutationally isomorphic to $H^{\Omega}$. 
\end{lemma}

%
%
%

The following result \cite[Theorem 3]{DYE} is useful for the characterisation of $\Omega^{\varepsilon}(V,Q).$

\begin{lemma} \label{DicksonInv}
Let $Q$ be a nondegenerate quadratic form on $V$ and let $I$ be the identity element of $\GL(V)$. Then an element $h \in O(V,Q)$ lies in $\Omega(V,Q)$ if and only if  $\mathrm{rk}(h-I)$ (the rank of the linear transformation) is even.
\end{lemma}

\begin{lemma}\label{lem:sprank}
With the notation of Construction~$\ref{con:gen}$ for $M=\Sp(2d,2)$,  $H^\Omega$ has rank $4.$

\end{lemma}
\begin{proof}
 Let $\alpha =R \in \Omega$, and let $\sigma=Q$ be the block of $\Sigma$ containing $\alpha$, so (identifying $\Sigma$ with the corresponding partition in $\Omega$) $\sigma = \{R, Rt\}$ where $t \in (M_{\sigma} \setminus R).$ 
By Corollary~\ref{cor:Rtr},  $R$, and hence also $H_\alpha$, is transitive on $\Sigma\setminus\{\sigma\}$.
Thus either $H_\alpha$ is transitive on $\Omega\setminus\sigma$ and $H^\Omega$ has rank 3, or $H_\alpha$ has two equal length orbits in $\Omega\setminus\sigma$ and $H^\Omega$ has rank $4$.  We  show that the latter holds.

 Let $\sigma'\in\Sigma\setminus\{\sigma\}$, and write $\sigma'=Q^g$ for some $g\in M$.
The stabiliser of  $Q^g$ is $M_{\sigma'}=M_{\sigma}^g$. 
Le $\alpha':=\alpha g\in\sigma'$, so $M_{\alpha'}=R^g$.
We claim that $H_{\alpha, \sigma'}$ stabilises $\alpha'.$ Observe that
$$H_{\sigma} = C \times M_{\sigma} \text{ and } H_{\alpha} = R \cup Rtc$$
where $\langle c \rangle= C.$ Indeed, the first equality follows since $C$ stabilises $\sigma$; the second equality follows since $C$ acts regularly on $\sigma.$
Now
\begin{align*}
H_{\alpha, \sigma'} & =H_{\alpha} \cap H_{\sigma'}  = (R \cup Rtc) \cap (C \times M_{\sigma'}) \\
& = (R \cap (C \times M_{\sigma'})) \cup (Rtc \cap (C \times M_{\sigma'})) \\ 
& = (R \cap M_{\sigma'}) \cup (Rt \cap M_{\sigma'})c.
\end{align*}
Let $h \in R \cap M_{\sigma'}= M_{\alpha, \sigma'}$. By Lemma \ref{DicksonInv}, we obtain that $\mathrm{rk}(h-I)$ is even since $h \in R=\Omega(V,Q)$. However, $h\in M_{\sigma'}=O(V,Q^g)$ so using    Lemma \ref{DicksonInv}  for the form $Q^g$, we find that  $h \in \Omega(V,Q^g)=R^g=M_{\alpha'}$. So $R \cap M_{\sigma'}$ stabilises $\alpha'.$ 
Now consider $hc \in (Rt \cap M_{\sigma'})c$ with $h \in Rt \cap M_{\sigma'}.$ By Lemma \ref{DicksonInv}, $\mathrm{rk}(h-I)$ is odd since  $h\in Rt= O(V,Q) \setminus \Omega(V,Q).$ Hence, by Lemma \ref{DicksonInv}, $h \in O(V,Q^g) \setminus \Omega(V,Q^g)$, and therefore $h$
swaps the two points in $\sigma'$.
So $hc$ acts trivially on $\sigma'$ and stabilises $\alpha'.$ Thus $H_{\alpha, \sigma'}$ fixes $\alpha'$, and so $H_\alpha$ has two orbits on $\Omega\setminus\sigma$ and $H^\Omega$ has rank $4$.
\end{proof}

\subsection{Line $6$: plinth ${\rm Ree}(q), q>3$}

Let $M={\rm Ree}(q)$ with $q=3^{2a+1}>3.$ We summarise information from \cite{REE} and \cite[Chapter 13]{Car} that we use in our analysis. Let $X$ be the automorphism group of $M$, so $X= M \rtimes \langle \phi \rangle$ where $\phi$ is a field automorphism of $M$ of order $2a+1.$ Let $B= U \rtimes T$ be the standard  Borel subgroup of $M$ where $|U|=q^3$ and $T= \langle t \rangle$ is the standard maximal torus of order $q-1,$ so $t^{\phi}=t^3.$ Then there exists $n_0 \in M$ such that $B \cap B^{n_0} =T$ (in \cite{REE} $n_0$ is denoted by $\omega_0$). Moreover, $n_0^{\phi}=n_0$ and $\phi$ normalises $U,T$ and $B$.

Let $A= \langle t^2 \rangle$ be the subgroup of $T$ of index $2$   and let $R=(U \rtimes A).$ 
Denote $S=(U \rtimes A) \rtimes \langle \phi \rangle$ and consider the action of $X$ on  $\Omega = \{Sx \mid x \in X\}.$ It is straightforward to see that the action (by right multiplication) of $M$ on $\Omega$ is equivalent to the action of $M$ on $\{Rx \mid x \in M\}$, and that $M$ (and $X$) preserve the block system $\Sigma=\{ \{Sx,Stx\} \mid x \in X\}$ in $\Omega$. In particular, if $\sigma = \{S,St\} \in \Sigma$, then $M_{\sigma}=B,$ so $(X^{\Sigma}, (U\rtimes A))$ is a special pair by Theorem \ref{t:special}.  Let $C= \C_{\Sym(\Omega)}(M)$ so, by  Lemma \ref{lem3}, $|C|=2.$ Let $c \in \Sym(\Omega)$ be defined by the rule $c: Sx \mapsto Stx.$ It is easy to see that $c$ centralises $M$ and $|c|=2,$ so $C=\langle c \rangle.$ Let $\ff:=\langle C, X^{\Omega} \rangle \le \Sym(\Omega).$ Notice that $c$ centralises $\phi$. Indeed,  since $t^{\phi}=t^3$, $t^2\in S$  and $S\phi=S$,
$$(Sx)c \phi = (Stx) \phi = (S \phi) \phi^{-1}tx \phi = S t^{\phi} x^{\phi}= S t^{3} x^{\phi}= Stx^{\phi}=(Sx^{\phi})c=(Sx)\phi c. $$
Hence $\ff = C \times X$ and we obtain the following construction.

\begin{construction}\label{con:ree}
Let $M={\rm Ree}(q)$ with $q = 3^{2a+1} > 3$ acting on $\Sigma$ of degree $q^3+1$, $r=2$, and, for $\sigma\in\Sigma$, let $R$ be the unique index $2$ subgroup of $M_\sigma$ so that $(M^\Sigma,R)$ is a special pair.
Let $X= M \rtimes \langle \phi \rangle$ where $\phi$ is a field automorphism of $M$ of order $2a+1,$ let $S=R\rtimes \langle \phi \rangle$ and $\Omega = \{Sx \mid x \in X\}.$
Let  $\ff = \langle c\rangle \times X^\Omega$, where $c: Sx \mapsto Stx$ for $t$ a generator of the maximal torus of $M_\sigma$ (note $t\in M_\sigma\setminus R$). 
 Then $\ff^{\Omega} \in \PIT$ with plinth $M$ and  $C:=\C_{\ff}(M)=\langle c\rangle\cong C_2$. Also $\ff$ stabilises the partition  $\{ \{Sx,Stx\} \mid x \in X\}$ of $\Omega$, which we identify with $\Sigma$.
\end{construction}

 Note that Construction \ref{con:ree} is equivalent to Construction \ref{con:gen}. We state it in this form for convenience of the proof of the results below. First we verify assertions made in Construction~\ref{con:ree}, and show that all properly innately transitive groups arising from the construction  have rank $4$.

\begin{lemma}\label{lem:reerank}
With the notation of  Construction~$\ref{con:ree}$, assume that $G\leq \Sym(\Omega)$ is such that $C \times M \le G \le \ff$. Then the following hold.
\begin{enumerate}
    \item[$(a)$] $G^{\Omega} \in \PIT$ with plinth $M$ and $C=\C_{\Sym(\Omega)}(M)$;
    \item[$(b)$] $\widehat{\varphi}(G^{\Omega}) = (G^{\Sigma}, R^{\Sigma})$ is a special pair as in Line $6$ of Table \ref{t:allr}. In particular $\widehat{\varphi}(\overline{G}^\Omega)=(X^\Sigma, R)$;
    \item [$(c)$] $G^\Omega$ has rank $4$.
\end{enumerate} 
\end{lemma}
\begin{proof}
 $(a)$ Let $\alpha : = S \in \Omega,$ so $M_{\alpha}=R^{\Omega}$. We showed above that $c\in C$, and by its definition $c$ has order $2$, since $t^2\in S$. Also, by \cite[Theorem 3.2(i)]{PS}, $C\leq \C_{\Sym(\Omega)}(M) \cong N_M(R)/R = M_\sigma/R \cong C_2$, and so  $C=\C_{\Sym(\Omega)}(M)\cong C_2$. Thus $G$ normalises $M$ and $C$, and since $M$ is simple and transitive on $\Omega$ and $C$ is intransitive, 
it follows that $G^{\Omega} \in\PIT$ with plinth $M$, proving part $(a)$. 

$(b)$ By \cite[Theorem 3.2(ii)]{PS}, $\sigma=\alpha^C=\{S,St\}$ and (see also \eqref{e:sigmaaction}) $\sigma=\alpha^{N_M(R)}=\alpha^{M_\sigma}$, so we may identify $\Sigma$ with the set of $C$-orbits in $\Omega$. Thus $G^{\Sigma} \le X^{\Sigma}$ with equality if and only if $G=\ff.$ By the definition of $R$ and Theorem \ref{t:special}, $(G^{\Sigma},R)$ is a special pair as in Line 6 of Table \ref{t:allr}.

$(c)$  Recall that $\alpha=S \in \Omega$ and $\sigma:=\{S,St\} \in \Sigma$. Also recall from the first paragraph of this subsection that $n_0\in B=M$ satisfies $M\cap M^{n_0}=T$, and let $\alpha':=\alpha^{n_0}$ and $\sigma':=\sigma^{n_0}.$ 
By Corollary~\ref{cor:Rtr},  $R$, and hence also  $G_\alpha$, is transitive on $\Sigma\setminus\{\sigma\}$.
Thus either $G_\alpha$ is transitive on $\Omega\setminus \sigma$ and $G^\Omega$ has rank $3$, or $G_\alpha$ has two equal length orbits on $\Omega\setminus \sigma$ and $G^\Omega$ has rank $4$. We will show that the latter holds by showing that $\ff_\alpha$ has two equal length orbits on $\Omega\setminus \sigma$.
Now 
$$
\ff_{\sigma} = C \times X_{\sigma}=C \times (B \rtimes \langle \phi \rangle) \text{ and } \ff_{\alpha}= S \cup Stc,
$$
and so 
$$
\ff_{\sigma'}=(C \times X_{\sigma})^{n_0}=C \times X_{\sigma}^{n_0}=C \times (B^{n_0} \rtimes \langle \phi^{n_0} \rangle)=C \times (B^{n_0} \rtimes \langle \phi \rangle).
$$
Therefore, since $\ff_{\alpha}=S \cup Stc$, 
\begin{align*}
\ff_{\alpha, \sigma'} & = (S \cup Stc) \cap (C \times (B^{n_0} \rtimes \langle \phi \rangle))\\
                    & = \Bigl(S \cap (B^{n_0} \rtimes \langle \phi \rangle) \Bigr) \cup \Bigl(Stc \cap \bigl(C \times (B^{n_0} \rtimes \langle \phi \rangle)\bigr) \Bigr).
\end{align*}
First we note that  $S \cap (B^{n_0} \rtimes \langle \phi \rangle) = A \rtimes \langle \phi \rangle$ (since $B \cap B^{n_0} =T$) and hence stabilises $\alpha'.$ Secondly, 
$$
Stc \cap \bigl(C \times (B^{n_0} \rtimes \langle \phi \rangle)\bigr)= Stc \cap (B^{n_0} \rtimes \langle \phi \rangle)c= (A \rtimes \langle \phi \rangle)tc,
$$ 
and so this coset also stabilises $\alpha'.$ Thus  $\ff_{\alpha, \sigma'}$ stabilises $\alpha'$ and so is not transitive on $\sigma'$. Hence $\ff_\alpha$ has two equal length orbits in $\Omega\setminus\sigma$, and so does $G_\alpha$.  
Thus $G^\Omega$ has rank 4.
\end{proof}

Finally we show that all properly innately transitive groups corresponding to a special pair satisfying the conditions of Line 6 of Table~\ref{t:allr} arise from Construction~\ref{con:ree}.

\begin{lemma}\label{lem:sppirRee}
Suppose that $G_0^{\Omega_0} \le \Sym(\Omega_0)$ is such that $G_0^{\Omega_0} \in \PIT$ with plinth $M_0\cong {\rm Ree}(q)$ and $\widehat{\varphi}(G_0^{\Omega_0})$  a special pair satisfying the conditions of Line $6$ of Table $\ref{t:allr}$. Then 
 $G_0^{\Omega_0}$ is permutationally isomorphic to a subgroup $G^\Omega$ with $C \times M \le G \le \ff$ and $\ff, M, C, \Omega$ as in Construction~$\ref{con:ree}$. 
\end{lemma}

\begin{proof}
By Lemma \ref{lem3}, and identifying $M$-actions on $\{Rx \mid x \in M\}$ and  $\Omega=\{ Sx \mid x \in X\}$ we may assume that $M_0=M$, $\Omega_0=\Omega$, and that $(C \times M) \le G \le N_{\Sym(\Omega)}(M)$, where $G:=G_0^{\Omega}.$ 
By Theorem  \ref{t:special}, we can take $\alpha \in \Omega$ such that   $M_{\alpha}=R=(U \rtimes A)$, and, since $M\lhd \ff$, we have $\ff \le N_{\Sym(\Omega)}(M).$ Also 
$$|N_{\Sym(\Omega)}(M)/C| = |N_{\Sym(\Omega)}(M)^{\Sigma}|=|X| = |\ff/C|$$
where the second equality holds
by \cite[Table 7.4]{Cam99} since $M^{\Sigma}$ is $2$-transitive.
 Therefore  $N_{\Sym(\Omega)}(M)=\ff$ and the result follows.
\end{proof}


\section{Proof of main results} \label{sect:rank3}

Now we are ready to proof Theorems \ref{t:PITSP} and \ref{t:PIT3}.  We consult various pages of the Atlas, for example, the Atlas \cite[pp. 4, 18, 85, 88, 100, 123, 134]{Atlas}.  

\subsection*{Proofs of Theorems \ref{t:PITSP} and \ref{t:PIT3}}
Our two goals are to find all preimages in $\PIT$ of the special pairs given by Theorem~\ref{t:special}, and all groups in $\PIT_3$, that is to say, we will achieve {\bf Objectives 2 and 3}. Since  by Lemma \ref{lem2} each group in $\PIT_3$ is a preimage of a special pair, our strategy is to go through Table \ref{t:allr} line by line, finding the preimages in $\PIT$, up to permutational isomorphism, and obtaining their ranks (or at least deciding if they have rank $3$ or not). In the following the Line numbers refer to Table~\ref{t:allr}. Recall that, by Theorem~\ref{t:special}, in each Line the subgroup $R$ is the unique normal subgroup of $M_\sigma$ of index (the appropriate prime) $r$.

\medskip

{\bf Line 1.} First we construct a preimage in $\PIT$ of the special pair for Line 1. Let $M=A_5$ acting naturally  on a set $\Sigma$ of size $5$ and,  for  $\sigma\in\Sigma$, let $R=V_4$ be the unique index $3$ normal subgroup of $M_\sigma=A_4$, and let $\Omega=\{Rx \mid x\in M\}$ with $M$ acting by right multiplication. Let $\alpha=R\in\Omega$ so that $M_\alpha=R$. By \cite[Theorem 3.2]{PS}, $C:=\C_{\Sym(\Omega)}(M)\cong N_M(M_\alpha)/M_\alpha=N_M(R)/R\cong C_3$. Let $X:=N_{\Sym(\Omega)}(M)$. Then it can be checked (by hand or computationally) that $X=(C\times M).2\cong \GammaL(2,4)$ has rank $3$ on $\Omega$, and that $X^\Sigma=X/C\cong S_5$ and $R^\Sigma\cong R$. Also it is straightforward to check that $(X^\Sigma, R^\Sigma)$ is a special pair (using Definition~\ref{def1}) satisfying the conditions of Line 1, and that $X^\Omega\in\PIT$ with $\widehat{\varphi}(X^\Omega)=(X^\Sigma, R^\Sigma)$. By Lemma~\ref{lem3}$(b)$, it follows that all preimages in $\PIT$ of $(X^\Sigma, R^\Sigma)$ are permutationally isomorphic to $X^\Omega$. Therefore part $(a)$ holds in Theorem \ref{t:PITSP} and the fourth line holds in Table \ref{t:allrk3}.

\medskip

{\bf Line 7.} Let  $M=\PSL(2,8)={\rm Ree}'(3)$ acting on $\Sigma$ such that $N_{\Sym(\Sigma)}(M)={\rm Ree(3)}$ is $2$-transitive. Then  $M_\sigma=C_9\rtimes C_2$. Let $R=C_9$, and 
consider the action of $M$ on $\Omega=\{Rx \mid x\in M\}$. 
If $G \le \Sym (\Omega)$ with a special pair as in Line $7$ of Table \ref{t:allr}, then by Lemma $\ref{lem3}(b)$ we may assume that 
$$
C \times M \le G \le N_{\Sym(\Omega)}(M)
$$ where $C=\C_{\Sym(\Omega)}(M)$ is as in Construction~\ref{con:gen}. Moreover, by Table \ref{t:allr}, $G^{\Sigma}= {\rm Ree(3)}$.  It is easy to verify computationally  that $N_{\Sym(\Omega)}(M)/C \cong {\rm Ree(3)},$ so $G=N_{\Sym(\Omega)}(M)$, and that $G^\Omega$ has rank 4. Therefore part $(a)$ holds in Theorem \ref{t:PITSP} and there is no contribution to Table \ref{t:allrk3}.

\medskip

{\bf Lines 3, 8, 9, 10.}
Here $M=\PSL(3,2)$, $\rm{M}_{11}$, $\rm{HS} $ or $\rm{Co}_3$, respectively. 
Also,  $M_\sigma= N.2$ where $N=A_4, A_6, \PSU(3,5)$ or  $\rm{McL}$, respectively.
Moreover $R=N$, $C=M_\sigma/R$,  $r=2$, and  $N_{\Sym(\Sigma)}(M)=M$ in each case. Thus, by Lemma \ref{lem3}(b) and Corollary \ref{lem:uniq}, up to permutational isomorphism, $G=C\times M$ and $\Omega=\{Rx \mid x\in M\}$. 
It is verified computationally that this action is rank 3 when $M \in \{\PSL(3,2), \rm{M}_{11}\}$ and rank 4 in the other two cases.
Therefore part $(b)$ holds in Theorem \ref{t:PITSP} and the  lines 5 and 6 hold in Table \ref{t:allrk3}.

\medskip

{\bf Line 5.} Here $M=\Sp(2d,2)$. Also $N_{\Sym(\Sigma)}(M)=M$, and so, by Lemma \ref{lem3}(b) and Corollary \ref{lem:uniq}, up to permutational isomorphism, $G=C\times M$ and $\Omega=\{Rx \mid x\in M\}$ (see also Lemma \ref{lem:spaction}). By Lemma \ref{lem:sprank}, this group has rank $4$.
Therefore part $(b)$ holds in Theorem \ref{t:PITSP} and 
there is no contribution to Table \ref{t:allrk3}.
\medskip

{\bf Lines 2, 4,  6.} For these Lines, such a group $G^\Omega\in\PIT$ exists by Lemmas~\ref{lem:sppirSL},~\ref{lem:sppirSU}, or~\ref{lem:reerank}, respectively. Each group $G^\Omega$ is as in part $(c)$  of Theorem \ref{t:PITSP}, by Lemmas~\ref{lem:pslaction},~\ref{lem:psUaction}, or~\ref{lem:sppirRee}, respectively.
By Lemma~\ref{lem:reerank}, each group $G$ for Line 6 of Table~\ref{t:allr} has rank $4$, and so the entry in Table~\ref{t:obj2} is correct and this case provides  no contribution to Table \ref{t:allrk3}.   

For Line 2, 
 $G^\Omega\in\PIT_3$ if and only  condition (b)(i) or (b)(ii) of Lemma~\ref{lem:sppirSL} holds. In the former case, line 1 of Table~\ref{t:allrk3} holds (note one of the conditions of Table \ref{t:allr} was omitted in Table~\ref{t:allrk3} since it is redundant as shown in Remark \ref{re:Table1}(ii))).
 In the latter case line $2$ of Table~\ref{t:allrk3} holds. In particular the entry in Table~\ref{t:obj2} is correct.
 
 For Line 4, 
 $G^\Omega\in\PIT_3$
  if and only if the conditions of  Line 4 of Table~\ref{t:allr} hold, by Lemma~\ref{lem:sppirSU}.
  Thus line $3$ of Table~\ref{t:allrk3} holds (again we omitted a redundant condition), and the entry in Table~\ref{t:obj2} is correct. This concludes the proof of Theorems \ref{t:PITSP} and \ref{t:PIT3}.

\medskip

We finish with a remark justifying the fact that lines 1, 2 and 3  of Table $\ref{t:allrk3}$ each yield infinitely many examples of rank $3$ groups in Theorem~\ref{t:PIT3}.

\begin{remark}\label{rem:infty}
{\rm 
In lines $1$ and $2$  of Table $\ref{t:allrk3}$, we may take $r=2$ and $q$ to be a prime such that $q\equiv 1\pmod 4$, of which there are infinitely many by Dirichlet's Theorem, (so $a=1$, and also $j=1$ in line 1). Then we may take $G$  to be the group $\ff$ in Construction~\ref{con:psl}, where $d$ is either $2$ or odd. We note that $\ff^\Sigma= \PGL(d,q)$ in both lines, and when $d$ is odd, $\ff= C_2\times \PGL(d,q)$ by Remark~\ref{rem:psl}.

In line $3$  of Table $\ref{t:allrk3}$, we may take $r=3$ and $q=q_0^2$ with the prime $q_0\equiv 2\pmod 3$, of which again there are infinitely many by Dirichlet's Theorem. Then we may take $G$ to be the group $\ff$ in Construction~\ref{con:psu}. We note that $\ff^\Sigma=\PGaU(d,q)$ and $|G^\Sigma/(G^\Sigma\cap\PGU(3,q))|=2a/j$, with $a=2, j=1$.
}
\end{remark}
\section{Appendix}\label{sect:appendix}


\bigskip

An incomplete list of innately transitive permutation groups of degree less than $60$ was included in John Bamberg's thesis \cite[Table 10.3]{B} in 2003. It focused on degrees less than $60$ for which imprimitive examples of innately transitive permutation groups exist. We found this list very helpful for studying small examples, as we developed our approach to the classification presented in this paper. In the course of our work, we discovered a small number of misprints and a missing case in Bamberg's table, and so we decided to produce a new list, generated and checked using Magma~\cite{magma}. As we rely on the data base of transitive permutation groups in \cite{magma}, our list covers degrees only up to $48$, and because there are more extensive lists of primitive and quasiprimitive groups \cite{B22, MAGMA1}, we include only properly innately transitive groups. In common with Bamberg's list, we provide in Table~\ref{t:allPIT}, for each $G^\Omega\in\PIT$ with $|\Omega|\leq 48$,
\begin{center}
    the degree $|\Omega|$, the order $|G|$, the $\mathrm{Rank}(G)$, the plinth $M\cong M^\Sigma$,\\
    and the Identification Number `Trans ID' for the group in Magma's database, 
\end{center}
where $\Sigma$ is the set of orbits in $\Omega$ of $\C_G(M)$, as discussed in Section~\ref{sect:intro}. At the time when John Bamberg's list was compiled the Identification Number was only available up to degree $35$. We refer to the transitive permutation group of degree $n$ with Trans ID $m$ as TR($n,m$).   To relate to the work of this paper we also give in Table~\ref{t:allPIT}
\begin{center}
    the size $r=|\C_G(M)|$, some information about the  quasiprimitive group $G^\Sigma$,\\ 
    and we indicate which lines correspond to special pairs.
\end{center}
We have not attempted to describe the structure of the group $G$ in detail as this can be explored computationally using Magma~\cite{magma} and the Identification number Trans ID.
However we note that, as discussed in Section~\ref{sect:intro}, $G^\Omega$ contains $M\times \C_G(M)$, and also $G^\Sigma = G/\C_G(M)$ and $M^\Sigma = M\C_G(M)/\C_G(M)\cong M$. Thus for those  lines where $M^\Sigma=G^\Sigma$,  the group $G^\Omega\cong M\times \C_G(M)$. 
There are a few lines which might cause the reader some confusion and we make a few additional comments about the third group of degree $40$ and the first four groups of degree $42$:
\begin{enumerate}
    \item The third line for degree $40$ is for the transitive group TG(40, 1191), which  contains the group TG(40, 587) as an index $2$ subgroup. For this group $G^\Omega$ we have $G^\Sigma\cong M_{10}$ and $(G^\Sigma)_\sigma\cong C_3^2:C_4$, and $(G^\Sigma)_\sigma$ has index $2$ in a maximal subgroup $L=C_3^2:Q_8$ of $M_{10}$ (the stabiliser in $M_{10}$ of a point of the projective line $\PG(1,9)$). Now $L$ has precisely three subgroups of index $2$. One of these subgroups is $L\cap M^\Sigma$, and $L\cap M^\Sigma$ is \emph{not} equal to $(G^\Sigma)_\sigma$ (since in the corresponding coset action for $M_{10}$, its socle $\PSL(2,9)\cong A_6$ has two orbits of length $10$). The actions for $M_{10}$ on the cosets of the other two index $2$ subgroups of $L$ are permutationally isomorphic (and are interchanged by an outer automorphism of $M_{10}$). Thus either of these coset actions may be taken for $G^\Sigma$. We note in addition that this also explains why we do not have an additional line in degree $40$ corresponding to a quasiprimitive  action $G^\Sigma$ of degree $20$ of $\Aut(A_6)$. 
    
    \item The first two lines for degree $42$ are for the transitive groups TG(42, 79) and TG(42, 80). In both lines the group $M\cong M^\Sigma=G^\Sigma=\PSL(3,2)$ and $r=2$, and hence $G^\Omega= M\times C$ (the same abstract group) with $C=C_2$. The group action $M^\Sigma$ is quasiprimitive but imprimitive on the set of $21$ flags (incident point-line pairs) of the Fano plane $\PG(2,2)$, and hence the stabiliser $M_\sigma \cong D_8$ (for $\sigma\in\Sigma$). A point stabiliser $M_\alpha$ (for $\alpha\in\sigma$) has index $2$ in $M_\sigma$ and is $C_4$ or $C_2^2$ for the groups TG(42,79) or TG(42,80), respectively.
    
    \item The third line for degree $42$ is for the transitive group TG(42,131), which  contains the group TG(42,79) as an index $2$ subgroup. The actions of the plinth $M^\Omega$ are the same for these two groups, namely the right multiplication action of $\PSL(3,2)$ on the set of right cosets of a subgroup $C_4$.  For the group $G= {\rm TG}(42,103)$,  $G^\Sigma$ consists of all correlations (that is, all collineations and dualities) of $\PG(2,2)$, acting primitively on the set of 21 flags.
    
    \item The fourth line of degree $42$ is for the transitive group TG(42,103), and here $M^\Sigma=\PSL(3,2)$ of degree $|\Sigma|=42/r=14$.  Hence $M_\sigma$ (for $\sigma\in\Sigma$) has index $14$ in $M$. There are two conjugacy classes of such subgroups, all isomorphic to $A_4$, and the $M$-coset actions for these two classes are permutationally isomorphic. Since $M^\Sigma=G^\Sigma$ in this line, the group $G^\Omega\cong C_3\times M$.
\end{enumerate}

The small discrepancies between the properly innately transitive entries in  John Bamberg's list and those in Table~\ref{t:allPIT} for degrees up to $48$, are very minor, and are as follows.
\begin{enumerate} 
\item The last line of degree 20 groups in Table~\ref{t:allPIT}, namely TG(20,265), was missing in \cite[Table 10.3]{B}.
\item There was a typo in the Identification number for the seventh group of degree $28$ in \cite[Table 10.3]{B}, which was listed as transitive group TG(28, 282) and should have been TG(28, 281), which is the second group of degree $28$ in Table~\ref{t:allPIT}.
 \end{enumerate}

Finally we wish to acknowledge the work of others in enumerating the transitive permutation groups of small degrees which our computations for Table~\ref{t:allPIT} rely on. The transitive groups with degrees up to $15$ were determined by Gordon Royle \cite{R87}, Greg Butler \cite{B93} and Conway, Hulpke \& McKay \cite{CHM98}, while the transitive groups with degrees in the range $16$ to $30$ were enumerated by Alexander Hulpke \cite{Hul05}. John Cannon and Derek Holt \cite{CH08}  determined the transitive groups of degree $32$; Derek Holt and Gordon Royle  \cite{HR20} determined those with degrees in the range $33$ to $47$, and Holt, Royle \& Tracey \cite{HRT22}  determined the transitive groups of degree $48$.


\begin{longtable}{|c|c|c|c|c| p{1.3cm} | p{3.5cm}|p{3.5cm}|}
\hline
 $|\Omega|$&Trans ID&$|G|$&$r$&Rank$(G)$&special pair &$M^\Sigma$ &$G^\Sigma$\\ 
\hline
  $12$&$ 76$&$ 120$&$ 2$&$ 4$&yes&$ \PSL(2,5)$&$ \PSL(2,5)$ \\
  $12$&$ 124$&$ 240$&$ 2$&$ 3$&yes&$ \PSL(2,5)$&$ \PGL(2,5)$ \\
  \hline
  $14$&$ 17$&$ 336$&$ 2$&$ 3$&yes&$ \PSL(3,2)$&$ \PSL(3,2)  $ \\
  \hline
  $15$&$ 15$&$ 180$&$ 3$&$ 4$&no&$A_5$&$A_5  $ \\
  $15$&$ 21$&$ 360$&$ 3$&$ 3$&yes&$ A_5$&$ S_5  $ \\
  \hline
 $ 20$&$ 36$&$ 120$&$ 2$&$ 6$&no&$ A_5$ on pairs&$ A_5$ on pairs \\
  $20$&$ 65$&$ 240$&$ 2$&$ 5$&no&$ A_5$ on pairs&$ S_5$ on pairs \\
  $20$&$ 152$&$ 720$&$ 2$&$ 4$&yes&$  \PSL(2,9)$&$ \PSL(2,9)  $ \\
  $20$&$ 197$&$ 1440$&$ 2$&$ 3$&yes&$ \PSL(2,9)$&$ M_{10}  $ \\
  $20$&$ 198$&$ 1440$&$ 2$&$ 4$&yes&$ \PSL(2,9)$&$\PSigmaL(2,9)$ \\
  $20$&$ 200$&$ 1440$&$ 2$&$ 3$&yes&$ \PSL(2,9)$&$ \PGL(2,9)  $ \\
  $20$&$ 265$&$ 2880$&$ 2$&$ 3$&yes&$ \PSL(2,9)$&$\PGammaL(2,9)  $ \\
  \hline
  $22$&$ 26$&$ 15840$&$ 2$&$ 3$&yes&$ M_{11}$&$ M_{11}  $ \\
  \hline
  $24$&$ 1355$&$ 504$&$ 3$&$ 4$&no& $  \PSL(2,7)$&$ \PSL(2,7)  $ \\
  $24$&$ 2668$&$ 1008$&$ 3$&$ 4$&no&$ \PSL(2,7)$&$ \PGL(2,7) $ \\
  \hline
  $26$&$ 47$&$ 11232$&$ 2$&$ 3$&yes&$\PSL(3,3)$&$ \PSL(3,3)  $ \\
  \hline
  $28$&$ 199$&$ 2184$&$ 2$&$ 4$&yes&$ \PSL(2,13)$&$ \PSL(2,13)  $ \\
  $28$&$ 281$&$ 4368$&$ 2$&$ 3$&yes&$ \PSL(2,13)$&$ \PGL(2,13)  $ \\
  \hline
  $30$&$ 29$&$ 120$&$ 2$&$ 10$&no&$ A_5$ impr action on $2|2|1$ partitions&$ A_5$ impr action on $2|2|1$ partitions\\
  $30$&$ 58$&$ 240$&$ 2$&$ 7$&no&$ A_5$ impr action on $2|2|1$ partitions&$ S_5$ impr action on $2|2|1$ partitions\\
  $30$&$ 179$&$ 720$&$ 2$&$ 5$&no&$ A_6$ on pairs&$ A_6$ on pairs \\
  $30$&$ 261$&$ 1440$&$ 2$&$ 5$&no&$ A_6$ on pairs&$ S_6$ on pairs\\
  \hline
  $36$&$ 5559$&$ 4896$&$ 2$&$ 4$&yes&$\PSL(2,17)$&$ \PSL(2,17)  $ \\
  $36$&$ 8345$&$ 9792$&$ 2$&$ 3$&yes&$\PSL(2,17)$&$ \PGL(2,17)  $ \\
  \hline
  $40$&$ 587$&$ 720$&$ 2$&$ 8$&no&$ A_6$ on $3$-sets&$ A_6$ on $3$-sets \\
  $40$&$ 1189$&$ 1440$&$ 2$&$ 6$&no&$ A_6$ on $3$-sets&$ S_6$ on $3$-sets\\
  $40$&$ 1191$&$ 1440$&$ 2$&$ 5$&no&$ A_6$ on $3$-sets&$M_{10}$ on $(C_3^2\,:\,C_4)$-cosets \\
  $40$&$ 1197$&$ 1440$&$ 4$&$ 6$&no&$\PSL(2,9)$&$\PSL(2,9) $ \\
  $40$&$ 2312$&$ 2880$&$ 4$&$ 4$&no&$ \PSL(2,9)$&$ M_{10} $ \\
 $ 40$&$ 2314$&$ 2880$&$ 4$&$ 5$&no&$\PSL(2,9)$&$ \PSigmaL(2,9)  $ \\
 $ 40$&$ 2323$&$ 2880$&$ 4$&$ 5$&no&$ \PSL(2,9)$&$ \PGL(2,9)  $ \\
  $40$&$ 5156$&$ 5760$&$ 4$&$ 4$&no&$ \PSL(2,9)$&$ \PGammaL(2,9)  $ \\
  \hline
 $ 42$&$ 79$&$ 336$&$ 2$&$ 10$&no&$ \PSL(3,2)$ on flags&$ \PSL(3,2)$ on flags \\
 $ 42$&$ 80$&$ 336$&$ 2$&$ 9$&no&$ \PSL(3,2)$ on flags&$ \PSL(3,2)$ on flags\\
 $42$&$ 131$&$ 672$&$ 2$&$ 7$&no&$ \PSL(3,2)$ on flags&correlations  of $\PG(2,2)$ on flags\\
 $ 42$&$ 103$&$ 504$&$ 3$&$ 7$&no&$ \PSL(3,2)$ on $A_4$-cosets&$ \PSL(3,2)$ on $A_4$-cosets\\
  
  $42$&$ 175$&$ 1008$&$ 6$&$ 5$&no&$ \PSL(3,2) $&$ \PSL(3,2)   $ \\
  $42$&$ 335$&$ 3276$&$ 3$&$ 4$&no&$  \PSL(2,13) $&$ \PSL(2,13)  $ \\
  $42$&$ 484$&$ 6552$&$ 3$&$ 4$&no&$ \PSL(2,13)$&$ \PGL(2,13)  $ \\
  $42$&$ 410$&$ 5040$&$ 2$&$ 5$&no&$ A_7$ on pairs&$ A_7$ on pairs \\  
  $42$&$ 550$&$ 10080$&$ 2$&$ 5$&no&$ A_7$ on pairs&$ S_7$ on pairs \\
\hline
\caption{Properly innately transitive groups $G$ of degree up to $48$}\label{t:allPIT}
\end{longtable}

\bibliographystyle{abbrv} 

\end{document}